\documentclass[11pt]{amsproc} 
\usepackage{geometry}
\usepackage[utf8]{inputenc}
\usepackage{amssymb,amsfonts,amsmath}
\usepackage{tikz,tikz-cd}
\usetikzlibrary{calc,shapes.geometric,arrows,decorations.pathreplacing}
\usepackage[colorinlistoftodos]{todonotes}
\usepackage{enumitem}
\usepackage{hyperref}

\newcommand{\rA}{\mathrm{A}}

\newcommand{\rE}{\mathrm{E}}
\newcommand{\rF}{\mathrm{F}}

\newcommand{\rH}{\mathrm{H}}

\newcommand{\rS}{\mathrm{S}}
\newcommand{\rT}{\mathrm{T}}

\newcommand{\cC}{\mathcal{C}}

\newcommand{\cH}{\mathcal{H}}

\newcommand{\cO}{\mathcal{O}}
\newcommand{\cP}{\mathcal{P}}

\newcommand{\cT}{\mathcal{T}}

\newcommand{\cX}{\mathcal{X}}

\newenvironment{smallpmatrix}{\left(\begin{smallmatrix}}{\end{smallmatrix}\right)}

\newcommand{\N}{\mathbb{N}}
\newcommand{\Q}{\mathbb{Q}}
\newcommand{\R}{\mathbb{R}}
\newcommand{\Z}{\mathbb{Z}}

\DeclareMathOperator{\Aut}{Aut}

\DeclareMathOperator{\Cay}{Cay}

\DeclareMathOperator{\rad}{rad}

\DeclareMathOperator{\GL}{GL}

\DeclareMathOperator{\SO}{SO}
\DeclareMathOperator{\Orth}{O}

\DeclareMathOperator{\PGL}{PGL}

\DeclareMathOperator{\PO}{PO}

\DeclareMathOperator{\Res}{Res}

\usepackage{amsmath}
\swapnumbers 
\newtheoremstyle{mythmstyle}
  {10pt} 
  {10pt} 
  {} 
  {\normalparindent} 
  {\bfseries} 
  {.} 
  {.5em} 
  {} 
\theoremstyle{mythmstyle}
\newtheorem{theorem}{Theorem}[section]
\newtheorem{lemma}[theorem]{Lemma}
\newtheorem{proposition}[theorem]{Proposition}
\newtheorem{corollary}[theorem]{Corollary}

\newtheorem{remark}[theorem]{Remark}

\newtheorem*{theorem*}{Theorem}
\newtheorem*{lemma*}{Lemma}
\newtheorem*{proposition*}{Proposition}
\newtheorem*{corollary*}{Corollary}
\newtheorem*{definition*}{Definition}
\newtheorem*{remark*}{Remark}
\newtheorem*{conjecture*}{Conjecture}
\newtheorem*{example*}{Example}
\newtheorem*{examples*}{Examples}

\makeatletter
\renewcommand\subsubsection{\@startsection{subsubsection}{2}%
  \normalparindent{.5\linespacing\@plus.7\linespacing}{-.5em}%
  {\normalfont\bfseries}}
\makeatother

\DeclareMathOperator{\Sph}{\rS}

\newcommand{\HH}{\mathrm{H}}
\newcommand{\NN}{\mathbb{N}}
\DeclareMathOperator{\HR}{\mathrm{HR}}
\DeclareMathOperator{\HRC}{\mathrm{HRC}}

\title[Constructing highly regular expanders from hyperbolic Coxeter groups]{Constructing highly regular expanders \\
from hyperbolic Coxeter groups}

\author{Marston Conder}
\address{M.C. and J.S.:\ Department of Mathematics, University of Auckland,\\ 38 Princes Street, Auckland 1010, New Zealand}
\email{m.conder@auckland.ac.nz,j.schillewaert@auckland.ac.nz}
\thanks{Grant support: M.C.\ by N.Z.\ Marsden Fund (project UOA1626), J.S.\ and M.C.\ by UoA (FRDF grant 3719917 `Geometry and symmetry'), A.L.\ by NSF (grant DMS-1700165) and ERC (Horizon 2020 programme, grant 692854), and F.T.\ by \href{https://www.frs-fnrs.be/en/}{FNRS} (CR FC 4057).}

\author{Alexander Lubotzky}
\address{A.L.:\ Einstein institute of mathematics
Edmund J. Safra campus of the Hebrew University
Givat Ram, Jerusalem 91904
Israel}
\email{alex.lubotzky@mail.huji.ac.il}

\author{Jeroen Schillewaert}

\author{Fran{\c c}ois Thilmany}
\address{F.T.:\ Institut de recherche en math{\' e}matique et physique, Universit{\' e} catholique de Louvain, 2 Chemin du Cyclotron, 1348 Louvain-la-Neuve, Belgique}
\email{francois.thilmany@uclouvain.be}

\dedicatory{Dedicated to John Conway (1937-2020) and Ernest Vinberg (1937-2020), for their phenomenal insights and outstanding contributions in the fields of algebra, combinatorics and geometry}

\date{September 2020}

\begin{document}

\begin{abstract}
A graph $X$ is defined inductively to be $(a_0,\dots,a_{n-1})$-regular if $X$ is $a_0$-regular and for every vertex $v$ of $X$, the sphere of radius $1$ around $v$ is an $(a_1,\dots,a_{n-1})$-regular graph. Such a graph $X$ is said to be highly regular (HR) of level $n$ if $a_{n-1}\neq 0$. 
Chapman, Linial and Peled \cite{ChapmanLinialPeled2019} studied HR-graphs of level $2$ and provided several methods to construct families of graphs which are expanders ``globally and locally''. They ask whether such HR-graphs of level $3$ exist. 

In this paper we show how the theory of Coxeter groups, and abstract regular polytopes and their generalisations, can lead to such graphs. Given a Coxeter system $(W,S)$ and a subset $M$ of $S$, we construct highly regular quotients of the 1-skeleton of the associated Wythoffian polytope $\cP_{W,M}$, which form an infinite family of expander graphs when $(W,S)$ is indefinite and $\cP_{W,M}$ has finite vertex links. 
The regularity of the graphs in this family can be deduced from the Coxeter diagram of $(W,S)$. The expansion stems from applying superapproximation to the congruence subgroups of the linear group $W$. 

This machinery gives a rich collection of families of HR-graphs, with various interesting properties, and in particular answers affirmatively the question asked in \cite{ChapmanLinialPeled2019}.
\end{abstract}

\maketitle

\section{Introduction}\label{sec:introduction}

A graph $X$ is defined inductively to be $(a_0,\dots,a_{n-1})$-regular if $X$ is $a_0$-regular and for every $v$ of $X$, the sphere of radius $1$ around $v$ is $(a_1,\dots,a_{n-1})$-regular. If $a_{n-1}\neq 0$, we will say that $X$ is a highly regular (HR) graph of level $n$. 
A convenient way to visualise such a graph is to think of the $n$-skeleton of the clique complex of $X$. This will be an $n$-dimensional simplicial complex, in which the $1$-skeleton of the link of every $i$-cell $(i=0,\dots,n-2)$ is an $a_{i+1}$-regular graph on $a_i$ vertices. 
If additionally the $1$-skeleta of all these links are connected, we say that $X$ is an $(a_0,\dots,a_{n-1})$-connected regular graph, and that $X$ is connected regular (HRC) of level $n$.

Motivated by questions related to PCP-theory, Chapman, Linial and Peled \cite{ChapmanLinialPeled2019} initiated a systematic study of HR-graphs of level $2$, that is, $(a,b)$-regular graphs. They were mainly interested in such graphs which are expanders ``globally and locally''. This means that the global graph
$X$ is an expander, but so are the links of vertices. Of course, for families of graphs for which the degree $a=a_0$ is constant, this simply means that the links are connected. 
They provided several methods to construct such families and raised the question if this can be done also for some triples $(a,b,c)$. The goal of this paper
is to show that the theory of Coxeter groups and their associated Wythoffian polytopes leads to a rich collection of HR-graphs.

The following theorem summarises our method. All the notions in the theorem will be explained in \S\ref{sec:polytopes-Coxeter-groups}.

\begin{theorem}\label{thm:Main}
Let $(W,S)$ be a Coxeter system, $M$ a subset of $S$ and $\cP_{W,M}$ the associated Wythoffian polytope. Suppose $(W,S)$ is indefinite, $\cP_{W,M}$ has finite vertex links, and the $1$-skeleton $X$ of $\cP_{W,M}$ is $(a_0, \dots, a_n)$-regular. Then there exists an infinite collection of finite quotients of $X$ by normal subgroups of $W$, which form a family of $(a_0, \dots, a_{n})$-regular expander graphs.
\end{theorem}

Let us say right away that the level of regularity of the graphs in Theorem \ref{thm:Main} is at most the rank of the Coxeter group from which they are derived, and usually it is much smaller.
Moreover, the HR-graphs provided by Theorem \ref{thm:Main} are expanders ``globally'', but their links may be disconnected when those of $\cP_{W,M}$ are. 
Thus Theorem \ref{thm:Main} gives a general scheme to construct highly regular expander graphs, but to apply it, one needs to find Wythoffian polytopes which are sufficiently (connected and) regular. This will be carried out in \S\ref{sec:highly-regular-polytopes}. 

In the particular case where $(W,S)$ is a string Coxeter system and $\cP_{W}=\cP_{W,M}$ is its universal polytope, the connected regularity of the 1-skeleton $X$ follows from that of $\cP_{W}$ itself by a straightforward argument (Lemma \ref{lem:higher-regularity}). 
Thus we obtain the following corollary to Theorem \ref{thm:Main}. 

\begin{corollary} \label{cor:Main-string}
Let $(W,S)$ be a string Coxeter system, and let $\cP_W$ be its universal polytope. Suppose $(W,S)$ is indefinite and $\cP_W$ has finite vertex links. Then there exists an infinite collection of finite quotients of the $1$-skeleton $X$ of $\cP_W$ by normal subgroups of $W$, which form a family of $(a_0, \dots, a_{n-1})$-connected regular expander graphs, where $n$ is the largest integer for which $\cP_W$ has a simplicial $n$-face and $a_i$ is the size of the link of any $i$-face of $\cP_W$ ($0 \leq i \leq n-1$). 
\end{corollary}
 
The example with the highest level of (connected) regularity that can be constructed directly from Corollary \ref{cor:Main-string} is a family of $(120,12,5,2)$-regular expander graphs, quotients of the 1-skeleton of the hyperbolic tessellation with diagram 
\begin{tikzpicture}[scale=0.7, transform shape]
\tikzstyle{every node}=[draw,solid,draw=black,fill=black,shape=circle,minimum size=.05cm,inner sep=.05cm]
\path[solid,draw=black,fill=white,thin] 
(0,0) node[double] (root 1) {} --++(.5,0) node (root 2) {} --++(.5,0) node (root 3) {} --++(.5,0) node (root 4) {} -- node[draw=none,fill=none,shape=circle,above]{5} ++(.5,0) node (root 5) {};
\end{tikzpicture}
(see \S\ref{subsec:order-5-4-simplex-honeycomb}). This example already answers the question of Chapman, Linial and Peled positively. 

\medskip

Many interesting examples of expanders graphs of connected regularity levels 3 and 4 arise from Coxeter systems affiliated to the exceptional types $\rE$ (see Table \ref{tab:Wythoffian-hyperbolic-honeycombs}). For instance, using Theorem \ref{thm:Main} we construct a family of $(2160,64,21,10)$-connected regular expander graphs as quotients of the 1-skeleton of the Wythoffian polytope with diagram 
\begin{tikzpicture}[scale=0.7, transform shape]
\tikzstyle{every node}=[draw,solid,draw=black,fill=black,shape=circle,minimum size=.05cm,inner sep=.05cm]
\path[solid,draw=black,fill=white,thin] 
(0,0) node[double] (root 1) {} --++(.5,0) node (root 3) {} --++(.5,0) node (root 4) {} --++(.5,0) node (root 5) {} --++(0,.5) node (root 2) {}
(root 5)
\foreach \i in {6,...,9}
{
--++(.5,0) node (root \i) {}
};
\end{tikzpicture}
, whose vertex links are of type $\rE_8$. 
For each $m \geq 10$, we construct another remarkable family of $(2^{m-2}, \frac{(m-1)(m-2)}{2}, 2(m-3))$-connected regular expanders as quotients of the polytope of type $\rE_m$ with diagram 
\begin{tikzpicture}[scale=0.7, transform shape]
\tikzstyle{every node}=[draw,solid,draw=black,fill=black,shape=circle,minimum size=.05cm,inner sep=.05cm]
\path[solid,draw=black,fill=white,thin] 
(0,0) node[double] (root 1) {} --++(.5,0) node (root 3) {} --++(.5,0) node (root 4) {} --++(0,.5) node (root 2) {}
(root 4) --++(.5,0) node (root 5) {};
\path[dash pattern=on 2pt off 1pt,draw=black,fill=white,thin]
(root 5) --++(.5,0) node (root 6) {};
\path[solid,draw=black,fill=white,thin]
(root 6)
\foreach \i in {7,...,7}
{
--++(.5,0) node (root \i) {}
};
\end{tikzpicture}. 
Its vertex links are $(m-1)$-demicubes. 
In fact, these families are respectively $(2160,64,21,10,5)$ and $(2^{m-2}, \frac{(m-1)(m-2)}{2}, 2(m-3),m-3)$-regular, but the last link is disconnected. 

\medskip

The following theorem sums up the most interesting examples (see \S \ref{sec:highly-regular-polytopes} for more).

\begin{theorem}\label{thm:levels} \mbox{}
\begin{enumerate}[label=(\alph*),leftmargin=*,align=left,itemsep=0ex,labelsep=0pt,topsep=0pt]
    \item There are infinitely many $(a_0,a_1,a_2) \in \mathbb{N}^3$ for which there exists an infinite family of $(a_0,a_1,a_2)$-connected regular expanders.
    \item For $(a_0,a_1,a_2,a_3) \in \{(120,12,5,2),(2160,64,21,10)\}$, there exists an infinite family of $(a_0,a_1,a_2,a_3)$-connected regular expanders. 
    \item For each $m \geq 5$, there is an infinite family of $(\binom{2m}{m}, m^2, 2(m-1), m-2, m-3, \dots,1)$-regular expander graphs, for which the spheres around $i$-cliques are connected for $0 \leq i \leq m$, $i \neq 3$. The sphere around a triangle is a disjoint union of two complete graphs on $m$ vertices.
\end{enumerate}
\end{theorem}

To obtain the arbitrarily high levels of regularity promised in Theorem \ref{thm:levels}(c), it is necessary to consider general Wythoffian polytopes (not just the regular ones). 

The Wythoffian polytopes $\cP_{m}$ we construct to this end are associated, for any $m \geq 5$, with the diagram 
\begin{tikzpicture}[scale=0.7, transform shape]
\tikzstyle{every node}=[draw,solid,draw=black,fill=black,shape=circle,minimum size=.05cm,inner sep=.05cm]
\path[solid,draw=black,fill=white,thin] 
(0,0) node (root 1) {} --++(.5,0) node (root 2) {};
\path[dash pattern=on 2pt off 1pt,draw=black,fill=white,thin]
(root 2) --++(.5,0) node (root 3) {};
\path[solid,draw=black,fill=white,thin] 
(root 3) --++(.5,0) node (root 4) {} --++(.5,0) node (root 5) {};
\path[dash pattern=on 2pt off 1pt,draw=black,fill=white,thin]
(root 5) --++(.5,0) node (root 6) {};
\path[solid,draw=black,fill=white,thin]
(root 6) --++(.5,0) node (root 7) {};
\path[solid,draw=black,fill=white,thin]
(root 4) --++(0,.5) node[double] (root 0) {};
\draw [decorate,decoration={brace,amplitude=.1cm},xshift=0pt,yshift=.1cm]
(-0.1,0) -- node[draw=none,fill=none,shape=ellipse,above]{$m-1$} (1.1,0);
\draw [decorate,decoration={brace,amplitude=.1cm},xshift=0pt,yshift=.1cm]
(1.9,0) -- node[draw=none,fill=none,shape=ellipse,above]{$m-1$} (3.1,0);
\end{tikzpicture}
(obtained by extending the $\rA_{2m-1}$ diagram in its middle by an edge labeled 3 and circling the added vertex). 
The 1-skeleton $X_m$ of $\cP_m$ is a $(\binom{2m}{m}, m^2, 2(m-1), m-2, m-3, \dots,1)$-regular graph, that is, has regularity level $m+1$. 
The link of any vertex in $\cP_m$ is an $m$-rectified $(2m-1)$-simplex, with diagram 
\begin{tikzpicture}[scale=0.7, transform shape]
\tikzstyle{every node}=[draw,solid,draw=black,fill=black,shape=circle,minimum size=.05cm,inner sep=.05cm]
\path[solid,draw=black,fill=white,thin] 
(0,0) node (root 1) {} --++(.5,0) node (root 2) {};
\path[dash pattern=on 2pt off 1pt,draw=black,fill=white,thin]
(root 2) --++(.5,0) node (root 3) {};
\path[solid,draw=black,fill=white,thin] 
(root 3) --++(.5,0) node[double] (root 4) {} --++(.5,0) node (root 5) {};
\path[dash pattern=on 2pt off 1pt,draw=black,fill=white,thin]
(root 5) --++(.5,0) node (root 6) {};
\path[solid,draw=black,fill=white,thin]
(root 6) --++(.5,0) node (root 7) {};
\draw [decorate,decoration={brace,amplitude=.1cm},xshift=0pt,yshift=.1cm]
(-0.1,0) -- node[draw=none,fill=none,shape=ellipse,above]{$m-1$} (1.1,0);
\draw [decorate,decoration={brace,amplitude=.1cm},xshift=0pt,yshift=.1cm]
(1.9,0) -- node[draw=none,fill=none,shape=ellipse,above]{$m-1$} (3.1,0);
\end{tikzpicture}
, and the 1-skeleton of this link (which is also the sphere of radius 1 around any vertex in $X_m$) is the Johnson graph $J(2m,m)$. The associated Coxeter system is indefinite because $m \geq 5$, hence Theorem \ref{thm:levels}(c) follows from applying Theorem \ref{thm:Main} to $\cP_{m}$.

\medskip

It should be pointed out that in order to obtain highly regular expanders from Theorem \ref{thm:Main}, the Wythoffian polytope needs to be chosen very carefully: low-dimensional faces need to be simplices, the link of a vertex needs to be highly regular and finite, while the associated Coxeter system must be indefinite. The fact that these conditions are difficult to satisfy simultaneously makes the polytopes described above (and in \S\ref{sec:highly-regular-polytopes}) very special. 

Another point deserves attention: the clique complexes of the quotient graphs obtained through Theorem \ref{thm:Main} are not quotients of $\cP_{W,M}$ itself, but rather of the subcomplex of $\cP_{W,M}$ consisting of its simplicial faces. 
This is particularly apparent for the polytope $\cP_m$ described above (which, for $m \geq 3$, has $4$-faces which are not simplices yet is regular of level $\geq 4$), and further reflects the subtlety of finding highly regular polytopes to which one can apply Theorem \ref{thm:Main}. 

While the regularity of the 1-skeleton $X$ is obtained from the geometry of $\cP_{W,M}$, proving that the quotient graphs are expander graphs requires arguments of a completely different nature. 
To this end, we use the fact that Coxeter groups are linear, and when $(W,S)$ is indefinite, that they are not virtually solvable. To such linear groups, one may apply the recent \emph{superapproximation} results (cf.~\cite{SalehiGolsefidy2017,SalehiGolsefidy2019}). Superapproximation means that the quotients of the Cayley graph of $(W,S)$ modulo congruence subgroups are expanders. 
These quotients are not highly regular graphs, but they are quasi-isometric to highly regular quotients of $X$ (see Lemma \ref{lem:polytope-quasi-isometric}), from which we deduce (in Proposition \ref{lem:comparing-expansion}) that the latter are also expanders. 
Along the way, we use Proposition \ref{lem:comparing-expansion} to deduce an interesting corollary on high-dimensional expanders which may be of independent interest.

\subsection{Outline}
After a short prologue (\S\ref{sec:expansion-regularity-connectivity}) on expansion, regularity and connectivity of graphs, we start  \S\ref{sec:polytopes-Coxeter-groups} by recalling some basic notions concerning (abstract) polytopes (\S \ref{subsec:polytopes-basic}).
We then discuss the regularity of the $1$-skeleta of polytopes (Lemma \ref{lem:higher-regularity}), followed by a brief introduction to Coxeter systems (\S \ref{subsec:Coxeter-systems}), their geometric representation (\S \ref{subsec:Coxeter-geometric-representation}) and the Wythoffian polytopes associated with them (\S \ref{subsec:Wythoffian-polytopes}). 

\smallskip

In \S \ref{sec:comparing-graphs}, we use the Coxeter complex (\S \ref{subsec:comparing-graphs}) to prove that the Cayley graph of $(W,S)$ and the $1$-skeleton $X$ of the associated polytope $\cP$ are quasi-isometric when $\cP$ has finite vertex links (Lemma \ref{lem:polytope-quasi-isometric}). 
We show how this quasi-isometry and the regularity of $\cP$ can be preserved when passing to finite quotients of $X$ (\S \ref{subsec:comparing-quotients} \& \S \ref{subsec:comparing-regularity}), and prove that the expansion of these quotients amounts to the expansion of the corresponding quotients of the Cayley graph (Proposition \ref{lem:comparing-expansion}).

\smallskip

In \S\ref{sec:density-superapproximation} we invoke the theorem of \cite{BenoistdelaHarpe2004} (Theorem \ref{thm:CoxeterDensity}) determining the Zariski closure of a Coxeter group in its geometric representation, to allow us to apply superapproximation \cite{SalehiGolsefidy2019} (Theorem \ref{thm:Superapproximation}) to indefinite Coxeter groups. 
On the side, we mention an interesting corollary (\ref{cor:superapproximation}) of the superapproximation theorem, of independent interest. 
With this at hand, we conclude the proof of Theorem \ref{thm:Main} in \S \ref{subsec:proof-Main}. 

\smallskip

In \S \ref{sec:HDX}, we present a quick application of Proposition \ref{lem:comparing-expansion} to high-dimensional expanders, of independent interest. We also interpret part of our results in the context of Garland theory. 

\smallskip

In \S \ref{sec:highly-regular-polytopes} we discuss highly regular hyperbolic tessellations. 
We start with regular tessellations (\S \ref{subsec:regular-hyperbolic-tesselations}), record the relevant ones in Table \ref{tab:regular-hyperbolic-honeycombs}, and explain why the list is so short in Remark \ref{rem:limited-regular-examples}. We describe the most noteworthy example, the order-$5$ $4$-simplex honeycomb, in \S \ref{subsec:order-5-4-simplex-honeycomb}. 
We then explain how to find highly regular graphs among tessellations of hyperbolic space by Wythoffian polytopes (\S \ref{subsec:Wythoffian-hyperbolic-tesselations}). We again record the most relevant examples, in Table \ref{tab:Wythoffian-hyperbolic-honeycombs}, and add one with arbitrarily high (but not connected) regularity in \S \ref{subsec:arbitrarily-high-regularity}. 

\smallskip

In \S \ref{sec:degree-parameters}, we use two standard graph product constructions to obtain infinitely many graphs of regularity level $n$ from a given one. The Cheeger-Buser inequalities guarantee that also the expansion property will be preserved (see Lemma \ref{graphproducts}). We discuss some obvious restrictions on regularity parameters in \S \ref{subsec:restrictions} and explain why finding good necessary and sufficient conditions on these parameters is a difficult problem. Finally, we state two open problems on the subject. 

\smallskip

We conclude the paper with a tribute to John Conway and Ernest Vinberg in \S \ref{sec:CV}.

\begin{remark}
While writing this paper, we learned that in a work in preparation \cite{FriedgutIluz2020}, Friedgut and Iluz develop a very different method to produce expanders of arbitrarily high connected regularity level; see \S\ref{subsec:related} for more details.
\end{remark}

\subsection{Acknowledgements}
All four authors thank the Margaret and John Kalman Trust for the financial support of the Michael Erceg Senior Visiting Fellowship at the University of Auckland (UoA) which A.L.\ was awarded. 
A.L.\ and F.T.\ thank UoA for its hospitatility. 
The present work grew out of their visit at UoA. 

We thank Michael Chapman and Ehud Friedgut for useful conversations on the topic of this paper. 

\section{Expansion, regularity and connectivity}\label{sec:expansion-regularity-connectivity}

\subsection{Expander graphs}\label{subsec:expansion}
A finite graph $X$ is said to be \emph{$\epsilon$-expanding}, or an \emph{$\epsilon$-expander}, if its \emph{Cheeger constant}
\[
h(X) = \min_{\emptyset \subsetneq V \subsetneq X}\frac{|\partial V|}{\min\left(|V|, |X \setminus V|\right)}
\]
is at least $\epsilon$. Here $\partial V$ denotes the edge-boundary of a set $V$ of vertices of $X$. 
Of course, any non-trivial finite connected graph $X$ is a $\frac{2}{|X|}$-expander, and a complete graph of any size is a $1$-expander. The relevance of this notion appears when one can bound the Cheeger constant from below independently of the size of the graph, while keeping the degree under control. 
A family $\cX$ of graphs is thus called a \emph{family of ($\epsilon$-)expanders (of degree $a$)} if there exists $\epsilon > 0$ and $a \in \N$ such that each graph $X \in \cX$ has maximum degree at most $a$ and Cheeger constant $h(X) \geq \epsilon$. 
We refer the reader to the survey \cite{HooryLinialWigderson2006} of Hoory, Linial and Widgerson for the history, theory and applications of expander graphs. 

\subsection{Higher regularity} \label{subsec:higher-regularity}
Let $X$ be a graph, and $V$ a set of vertices of $X$. 
(Our notation will not distinguish the graph $X$ and its underlying vertex set; the meaning should be clear from the context.) 
We define $\Sph_X(V)$ as the \emph{sphere of radius $1$ around $V$ in $X$}, that is, the full subgraph of $X$ induced by the set of vertices $\{x \in X \setminus V \mid \textrm{$x$ is adjacent to every vertex of $V$}\}$. 
If $V = \{v\}$ consists of a single vertex $v$, then we also write this as $\Sph_X(v)$, denoting the \emph{(punctured) neighbourhood of $v$ in $X$} in this case. 
Now, let $a_0, \dots, a_n$ be cardinals. Inductively, a graph $X$ is called \emph{$(a_0, \dots, a_n)$-regular} if $X$ is an $a_0$-regular graph and $\Sph_X(v)$ is an $(a_1, \dots, a_n)$-regular graph for every vertex $v$ of $X$. We call $a_i$ the \emph{$i$\textsuperscript{th} regularity degree}, and the largest integer $n+1$ for which $X$ is $(a_0, \dots, a_n)$-regular with $a_n \neq 0$ will be called the \emph{regularity level of $X$}. 

Equivalently, $X$ is $(a_0, \dots, a_n)$-regular if for each $i \in \{0,1,\dots,n\}$, either there is no clique of size $i+2$ in $X$ and $a_i = 0$, or otherwise for every clique $C$ of size $i$ in $X$, the subgraph $\Sph_X(C)$ is $a_{i}$-regular. 
(By convention, the empty graph $\emptyset$ is $0$-regular, but not $d$-regular for any $d \geq 1$, and $\Sph_X(\emptyset) = X$.) 
Next, the \emph{clique complex $X^{\mathrm{cl}}$} of a graph $X$ is the simplicial complex whose $(i-1)$-simplices are the cliques of size $i$ in $X$ ($i \in \N$), with incidence between simplices being given by containment between the corresponding cliques in $X$. The graph $X$ identifies naturally with the $1$-skeleton of $X^{\mathrm{cl}}$ (that is, the graph consisting of all vertices and edges of $X^{\mathrm{cl}}$).
Thus $X$ is an $(a_0, \dots, a_n)$-regular graph if and only if its clique complex $X^{\mathrm{cl}}$ has the following regularity property: for $0 \leq i \leq n$, every $i$-simplex of $X^{\mathrm{cl}}$ is contained in exactly $a_i$ $(i+1)$-simplices; or equivalently (when $n \geq 1$): for $0 \leq i \leq n-1$, the 1-skeleton of the link of any $i$-simplex of $X^{\mathrm{cl}}$ is an $a_{i+1}$-regular graph on $a_i$-vertices. 

\subsection{Connected regularity} \label{subsec:connected-regularity}

In order for a regular graph to have not only global but also local expansion, the links must remain connected. So in a given $(a_0, \dots, a_n)$-regular graph $X$, if for $0 \leq i \leq n$ the sphere around every $i$-clique (or equivalently, the link of every $(i-1)$-simplex in $X^\textrm{cl}$) is connected, we call $X$ \emph{$(a_0, \dots, a_n)$-connected regular}. (In particular, $X$ itself should be connected.)
When there is no need to specify the regularity degrees, we say that $X$ has \emph{connected regularity level} $n+1$, or in short that $X$ is HRC of level $n+1$.

\section{Polytopes and their groups of symmetry}
\label{sec:polytopes-Coxeter-groups}

\subsection{Basics on polytopes} \label{subsec:polytopes-basic}

We begin by recalling (some of) the basic definitions concerning abstract polytopes following \cite{McMullenSchulte2002}. For more on polytopes, reflection groups and an explanation of the terminology used in what follows, we refer the reader to \cite{Coxeter1948}, \cite{McMullenSchulte2002} and \cite{Bourbaki2007}.

\medskip 

An abstract polytope $\cP$ of {\em (finite) rank $n$} is a poset, whose elements are called {\em faces}, satisfying properties (P1)-(P4) below. Two faces $F$, $G$ of $\cP$ are called \emph{incident} if $F\leq G$ or $G\leq F$.
\begin{itemize}
    \item[(P1)] $\cP$ contains a least face $F_{-1}$ and a greatest face $F_n$ (the {\em improper} faces).
    \item[(P2)] Each \emph{flag} (each totally ordered subset of $\cP$ of maximal length) has length $n+1$, that is, contains exactly $n + 2$ faces including $F_{-1}$ and $F_n$.
\end{itemize}

For any two faces $F$ and $G$ with $F\leq G$, we call the poset $G/F:=\{H\in \cP \mid F\leq H \leq G\}$ a \emph{section} of $\cP$.
We will often identify a face $F$ of $\cP$ with the section $F/F_{-1}$.  
The section $F_n/F$ is called the {\em link} of the face $F$.
If the rank of $F/F_{-1}$ is $i$, then $F$ is called an {\em $i$-face}. It is customary to call $0$-faces, $1$-faces and $(n-1)$-faces of $\cP$ respectively vertices, edges and facets. 

A poset $\cP$ satisfying (P1) and (P2) is said to be {\em connected} if either $n\leq 1$, or $n\geq 2$ and for any two proper faces $F$, $G$ there exists a finite sequence of proper faces $F=H_0,H_1,\dots,H_k=G$ such that $H_{i-1}$ and $H_i$ are incident for $i=1,\dots,k$. 

\begin{itemize}
    \item[(P3)] Every section of $\cP$ is connected.
\end{itemize}   

Given a poset $\cP$ with properties (P1) and (P2), two flags of $\mathcal{P}$ are called {\em adjacent} if they differ in exactly one face. 
Then $\mathcal{P}$ is called {\em flag-connected} if any two distinct flags $\Phi$ and $\Psi$ of $\mathcal{P}$ can be joined by a sequence of flags $\Phi=\Phi_0,\Phi_1,\dots,\Phi_{k-1},\Phi_k = \Psi$ such that $\Phi_{j-1}$ and $\Phi_j$ are adjacent for $j=1,\dots,k$. 
If a poset $\cP$ satisfies (P1) and (P2), then (P3) is equivalent to the a priori stronger condition that every section of $\cP$ is flag-connected. 

The last requirement connects abstract polytopes to traditional polytopes.
\begin{itemize}
    \item[(P4)] For each $i \in \{0,\dots,n-1\}$, if $F$ and $G$ are incident faces of $\cP$ of ranks $i-1$ and $i+1$ respectively, then there are precisely two $i$-faces $H$ of $\cP$ such that $F<H<G$.
\end{itemize}
It is an easy exercise to check that sections (in particular, faces and links) of abstract polytopes are again abstract polytopes. 

\begin{remark}
In this paper, we manipulate three different kind of \emph{links}, in three different classes of objects: spheres around cliques in a graph, (simplicial) links around simplices in a simplicial complex, and (polytopal) links around faces in a polytope. 
Even though the first two notions agree when the simplicial complex in question is the clique complex of a graph (cf.~\S \ref{subsec:higher-regularity}), and the last two agree up to a certain rank in a polytope with only simplicial faces up to that rank (as we will use in the proof of Lemma \ref{lem:higher-regularity}), they differ in general. 
For instance, when a polytope $\cP$ has a $2$-face $F$ which is not a triangle, the sphere in the $1$-skeleton of $\cP$ around a vertex $v$ of $F$ is a proper spanning subgraph of the $1$-skeleton of the (polytopal) link of $v$ in $\cP$. 
Care thus has to be taken when discussing links, and the context should make clear which notion of link is involved. 
This is particularly relevant for the second half of \S \ref{sec:highly-regular-polytopes}, where non-regular polytopes are considered. 
\end{remark}

Next, we prove an easy lemma that gives the regularity of the 1-skeleton of (sufficiently) regular polytopes, and then recall some basic facts about Coxeter groups and their associated polytopes. 

\begin{lemma} \label{lem:higher-regularity}
Let $\cP$ be an abstract polytope, and let $X$ denote the $1$-skeleton of $\cP$ (the graph consisting of the vertices and edges of $\cP$). 
Let $k$ be the largest integer for which $\cP$ has a $k$-face which is a simplex, and suppose that $\Aut \cP$ acts transitively on the $i$-faces of $\cP$ for $0 \leq i \leq n$. 
Then $X$ is a $(a_0, \dots, a_{\min(k,n)})$-regular graph, where $a_i$ is the number of simplicial $(i+1)$-faces containing a given $i$-face of $\cP$. 
Moreover, $X$ is $(a_0, \dots, a_{\min(k,n)-1})$-connected regular.
\end{lemma}
\begin{proof}
Set $k' = \min(k,n)$. By assumption, all $k'$-faces of $\cP$ are simplices, and the $k'$-skeleton of $\cP$ (the poset consisting of the $k'$-faces of $\cP$ and their subfaces) is a simplicial complex. 
This simplicial complex coincides with the $k'$-skeleton of the clique complex $X^{\mathrm{cl}}$ of $X$. 
Indeed, if an $i$-face of $\cP$ is a simplex, then its vertices obviously form a clique in $X$. Conversely, for $i \leq n$, the hull of any clique $\{x_0, \dots, x_i\}$ in $X$ is an $i$-face of $\cP$, which is necessarily the simplex with vertex set $\{x_0, \dots, x_i\}$. 

Using this observation and the transitivity assumption again, we deduce that $\Aut \cP$ acts transitively on the set of $i$-simplices of $X^{\mathrm{cl}}$ for $0 \leq i \leq k'$. Hence, for $0 \leq i \leq k'$, the number $a_i$ of $(i+1)$-simplices of $X^{\mathrm{cl}}$ containing a given $i$-simplex is independent of the choice of the latter. The first part of the lemma thus follows from the discussion in \S\ref{subsec:higher-regularity}, after noting that $a_i$ equals the number of {simplicial} $(i+1)$-faces containing any $i$-face of $\cP$. 

Finally, connectivity of the links of $(i-1)$-faces when all $(i+1)$-faces are simplicial (that is, when $i \leq \min(k,n)-1$) follows from the flag connectivity of $\cP$ (see \S\ref{subsec:polytopes-basic}). 
\end{proof}

Lemma \ref{lem:higher-regularity} obviously applies to \emph{regular polytopes} (whose automorphism group acts transitively on all flags) and to {\em chiral polytopes} (which are maximally symmetric by rotations, but admit no reflections). If $\cP$ is regular, $n$ can be taken to be the rank of $\cP$, so that $\min(k,n) = k$. Moreover, $a_k = 0$ and $a_{k-1} \geq 1$ by definition of $k$. If $\cP$ is chiral, then $n$ can be taken to be the rank of $\cP$ minus 1, so that again $\min(k,n) = k$.

\subsection{Coxeter systems} \label{subsec:Coxeter-systems}
Let $(W,S)$ be a \emph{(finitely generated) Coxeter system}. Recall that this means that $S$ is a finite set, and $W$ is the group with presentation
\[
W = \langle\, S \mid (st)^{m_{st}} = 1 \textrm{ for all } s,t \in S \,\rangle,
\]
where $m_{st} \in \{1,2, \dots, \infty \}$ for all $s,t \in S$, and satisfy $m_{st} = 1$ if and only if $s=t$. (It is understood that the relation $(st)^{m_{st}} = 1$ is omitted when $m_{st}=\infty$.) 
The $|S| \times |S|$ matrix $(m_{st})$ is called the \emph{Coxeter matrix of $(W,S)$}. The \emph{Coxeter diagram} of $(W,S)$ is the diagram consisting of $|S|$ vertices indexed by members of $S$, with two vertices $s$ and $t$ connected by an edge labelled $m_{st}$ if $m_{st} \geq 3$ (although the labels `3' are usually omitted). A group $W$ is called a \emph{Coxeter group} if it has a set of generators $S$ for which $(W,S)$ forms a (finitely generated) Coxeter system. 

A Coxeter system whose unlabelled diagram is a simple path is called a \emph{string} Coxeter system, or said to be \emph{of string type}. In that case, we will always assume that the elements of $S$ are indexed $s_0, \dots, s_{n-1}$ so that the edges of the diagram join $s_{i-1}$ to $s_{i}$ for $1 \leq i \leq n-1$, and the Coxeter matrix is usually abbreviated by its superdiagonal entries $[m_{0,1}, \dots, m_{n-2,n-1}]$. 

Tits \cite{Tits1960} showed how one can associate with every string Coxeter system $(W,S)$ endowed with one of the two possible indexations just described a regular polytope $\cP_W$ whose automorphism group is $W$. This polytope is called the \emph{universal polytope} for $(W,S)$, and is often denoted directly by its Schl{\" a}fli symbol $\{m_{0,1}, \dots, m_{n-2,n-1}\}$, because any other polytope with the same symbol is a quotient of $\cP_W$ (see \cite[Ch.~3D]{McMullenSchulte2002}). 

\begin{remark} \label{rem:regularity-Schlafli-symbol}
When $\cP$ is a universal $n$-polytope, the values of $a_0, \dots, a_k$ from Lemma \ref{lem:higher-regularity} can be deduced from the Schl{\" a}fli symbol $\{p_1, \dots, p_{n-1}\}$ of $\cP$ as follows. 
Let $F$ be a $i$-face of $\cP$, and let $L$ be the link of $F$ in $\cP$. The Schl{\" a}fli symbol of $F$ is then $\{p_1, \dots, p_{i-1}\}$, while that of $L$ is $\{p_{i+2}, \dots, p_{n-1}\}$. 
The integer $k$ defined above coincides with the smallest index $j$ for which $p_j \neq 3$ (with $k = n$ if $p_1 = \dots = p_{n-1} = 3$). 
For $0 \leq i \leq k-1$, the cardinal $a_i$ is the number of vertices in the universal polytope with symbol $\{p_{i+2}, \dots, p_{n-1}\}$, and $a_k = 0$. 
\end{remark}

\subsection{The geometric representation of a Coxeter group} \label{subsec:Coxeter-geometric-representation}
Let $(W,S)$ be a Coxeter system, and let $B$ be the bilinear form on $V = \R^{S}$ given with respect to the canonical basis $\{e_s :  s \in S\}$ by setting
\[
B(e_s,e_t) = -\cos(\pi/ m_{st}) \qquad \hbox{ for all } s,t \in S.
\] 
The \emph{geometric representation of $W$} on $V$ is defined by
\[
s(v) = v - 2{B(v,e_s)}e_s \qquad \hbox{ for all } v \in V, s \in S.
\]
It is a classical theorem of Tits that this representation is faithful. In fact, the dual space $V^*$ has a convex $W$-invariant cone, called the \emph{Tits cone}, which can be used to construct a geometric model for $W$ and a realisation of the associated Wythoffian polytopes, as we will discuss next. 

The image of $W$ under the geometric representation defined above (which we will identify with $W$) preserves the bilinear form $B$, and hence lies in the orthogonal group $\Orth_B$. The \emph{signature of $(W,S)$} is defined to be the signature of $B$; accordingly, we call $(W,S)$ \emph{definite}, \emph{semidefinite} or \emph{indefinite} when $B$ has the corresponding property. It is well known that $W$ is finite if and only if $(W,S)$ is (positive) definite.

\subsection{Wythoffian polytopes}\label{subsec:Wythoffian-polytopes}
Recall that with a Coxeter system $(W,S)$ and a distinguished subset $M$ of $S$ (usually circled on the Coxeter diagram of $(W,S)$, making it an \emph{adorned Coxeter diagram}), \emph{Wythoff's kaleidoscopic construction} associates an abstract polytope $\cP_{W,M}$ in roughly the following way. 
In the Tits cone of $(W,S)$, place a point $x_0$ on the intersection of the walls associated with $S \setminus M$ (the \emph{inactive mirrors} of the kaleidoscope), equidistantly from the walls associated with $M$ (the \emph{active mirrors}). The vertices of $\cP_{W,M}$ are the images of $x_0$ under $W$. The point $x_0$ is connected by an edge to each of its reflections under $S$ (or equivalently, $M$); and the edges of $\cP_{W,M}$ are the images of those edges under $W$. More generally, the images of $x_0$ under any standard parabolic subgroup of $W$ form a standard face of $\cP_{W,M}$, and arbitrary faces are obtained as images of standard faces under $W$. 
Incidence between two standard faces amounts to containment between the smallest standard parabolic subgroups corresponding to these faces (see below), and this incidence relation is propagated to $\cP_{W,M}$ by the action of $W$. 

Any polytope arising from this kaleidoscopic construction is called \emph{Wythoffian}. 
When $(W,S)$ is a string Coxeter system and $M = \{s_0\}$, then one obtains the universal polytope for $(W,S)$. 

The formalism of adorned Coxeter diagrams is very convenient to explore the geometry of a Wythoffian polytope $\cP_{W,M}$. Indeed, the shape of faces and links in $\cP_{W,M}$ can be read directly from the adorned diagram $\Delta$. 
The different faces of $\cP_{W,M}$ are also Wythoffian polytopes, whose adorned diagrams are all obtained through the following recipe. Remove from $\Delta$ any subset $R$ of vertices (not containing $M$), and let $\Delta_R$ be the union of the connected components of the resulting diagram which intersect $M$, with the vertices of $M$ remaining circled. Then $\Delta_R$ is the adorned diagram of (any image of) the standard face corresponding to the standard parabolic subgroup generated by the vertices of $\Delta_R$ (seen as vertices of $\Delta$). In particular, the rank of this parabolic subgroup coincides with the rank of the corresponding face. 
Although it works in greater generality, for our purposes we will restrict the discussion of the link to connected Coxeter diagrams $\Delta$ adorned with only one circle (that is, $(W,S)$ is irreducible and $M$ is a singleton). 
In this setting, the link of any vertex of $\cP_{W,M}$ is again a Wythoffian polytope, whose adorned diagram is obtained by removing $M$ and all edges connecting it from $\Delta$, and circling all vertices of this new diagram $\Delta \setminus M$ that were previously connected to $M$. 

\medskip

Recall that an abstract polytope $\cP$ is recursively called \emph{uniform} if $\Aut \cP$ acts transitively on the vertices of $\cP$, and every facet of $\cP$ is uniform. The discussion above shows why Wythoffian polytopes are {uniform}: their automorphism group acts transitively on vertices by construction, and their faces are all Wythoffian, hence uniform by induction on the rank. However, Wythoffian polytopes are generally not regular. 
Not all uniform polytopes are Wythoffian; a nice example of a non-Wythoffian uniform polytope is the `grand antiprism', discovered by Conway and Guy in 1965 \cite{Conway-Guy-1965}. It is unknown in general which fraction of the uniform polytopes the Wythoffian polytopes account for. 

Wythoff's kaleidoscopic construction was first described in these terms by Coxeter, in a series of papers \cite{Coxeter1934,Coxeter1940,Coxeter1985}. Unfortunately, the authors are not aware of any modern, textbook treatment of this material (except \cite{McMullenSchulte2002} in the special case of string Coxeter systems). 

\medskip

We conclude this subsection with a concrete example. For the remainder of this subsection, let $(W,S)$ be the semidefinite Coxeter system with adorned diagram
\begin{tikzpicture}[scale=0.7, transform shape]
\tikzstyle{every node}=[draw,solid,draw=black,fill=black,shape=circle,minimum size=.05cm,inner sep=.05cm]
\path[solid,draw=black,fill=white,thin] 
(0,0) node (root 1) {} -- node[draw=none,fill=none,shape=circle,above]{4} ++(.5,0) node[double] (root 2) {} --++(.5,0) node (root 3) {} -- node[draw=none,fill=none,shape=circle,above]{4} ++(.5,0) node (root 4) {};
\end{tikzpicture}
, label its vertices $\{s_0, \dots, s_3\}$ from left to right (so that $M = \{s_1\}$), and let $\cP_{W,M}$ be the associated Wythoffian polytope. 
The link of a vertex in $\cP_{W,M}$ has diagram
\begin{tikzpicture}[scale=0.7, transform shape]
\tikzstyle{every node}=[draw,solid,draw=black,fill=black,shape=circle,minimum size=.05cm,inner sep=.05cm]
\path[solid,draw=black,fill=white,thin] 
(0,0) node[double] (root 1) {} (1,0) node[double] (root 3) {} -- node[draw=none,fill=none,shape=circle,above]{4} ++(.5,0) node (root 4) {};
\end{tikzpicture}
, hence is the product of an edge 
(\begin{tikzpicture}[scale=0.7, transform shape]
\tikzstyle{every node}=[draw,solid,draw=black,fill=black,shape=circle,minimum size=.05cm,inner sep=.05cm]
\path[solid,draw=black,fill=white,thin] 
(0,0) node[double] (root 1) {};
\end{tikzpicture}) 
and a square 
(\begin{tikzpicture}[scale=0.7, transform shape]
\tikzstyle{every node}=[draw,solid,draw=black,fill=black,shape=circle,minimum size=.05cm,inner sep=.05cm]
\path[solid,draw=black,fill=white,thin] 
(1,0) node[double] (root 3) {} -- node[draw=none,fill=none,shape=circle,above]{4} ++(.5,0) node (root 4) {};
\end{tikzpicture}); 
in other words, it is a square prism. The different proper faces of $\cP_{W,M}$ can be listed following the recipe above, by removing in turn $s_0$, $s_3$, $s_2$, $\{s_0, s_3\}$, and $\{s_0, s_2\}$ from the diagram. The resulting possibilities are respectively an octahedron
(\begin{tikzpicture}[scale=0.7, transform shape]
\tikzstyle{every node}=[draw,solid,draw=black,fill=black,shape=circle,minimum size=.05cm,inner sep=.05cm]
\path[solid,draw=black,fill=white,thin] 
(0.5,0) node[double] (root 2) {} --++(.5,0) node (root 3) {} -- node[draw=none,fill=none,shape=circle,above]{4} ++(.5,0) node (root 4) {};
\end{tikzpicture}), 
a cuboctahedron
(\begin{tikzpicture}[scale=0.7, transform shape]
\tikzstyle{every node}=[draw,solid,draw=black,fill=black,shape=circle,minimum size=.05cm,inner sep=.05cm]
\path[solid,draw=black,fill=white,thin] 
(0,0) node (root 1) {} -- node[draw=none,fill=none,shape=circle,above]{4} ++(.5,0) node[double] (root 2) {} --++(.5,0) node (root 3) {};
\end{tikzpicture}), 
a square
(\begin{tikzpicture}[scale=0.7, transform shape]
\tikzstyle{every node}=[draw,solid,draw=black,fill=black,shape=circle,minimum size=.05cm,inner sep=.05cm]
\path[solid,draw=black,fill=white,thin] 
(0,0) node (root 1) {} -- node[draw=none,fill=none,shape=circle,above]{4} ++(.5,0) node[double] (root 2) {};
\end{tikzpicture}), 
a triangle
(\begin{tikzpicture}[scale=0.7, transform shape]
\tikzstyle{every node}=[draw,solid,draw=black,fill=black,shape=circle,minimum size=.05cm,inner sep=.05cm]
\path[solid,draw=black,fill=white,thin] 
(0,0) node[double] (root 2) {} --++(.5,0) node (root 3) {};
\end{tikzpicture}), 
and with no surprise, an edge
(\begin{tikzpicture}[scale=0.7, transform shape]
\tikzstyle{every node}=[draw,solid,draw=black,fill=black,shape=circle,minimum size=.05cm,inner sep=.05cm]
\path[solid,draw=black,fill=white,thin] 
(0,0) node[double] (root 1) {};
\end{tikzpicture}). 
With this information, it is not unreasonable to guess that the polytope $\cP_{W,M}$ is indeed the rectified cubic euclidean honeycomb. 

\section{From the Cayley graph of $(W,S)$ to an associated Wythoffian polytope} \label{sec:comparing-graphs}

In this section, we compare the Cayley graph and any Wythoffian polytope $\cP_{W,M}$ associated with a Coxeter system $(W,S)$, with the aim of constructing highly regular finite quotients of $\cP_{W,M}$ when $\cP_{W,M}$ has finite vertex links, which can be made arbitrarily large if $W$ is infinite.

\subsection{Comparing graphs} \label{subsec:comparing-graphs}
Let $(W,S)$ be a Coxeter system and $M$ a subset of $S$. 
The appropriate space in which to study $\Cay(W,S)$ and $\cP_{W,M}$ simultaneously is the Coxeter complex $\cC$ of $(W,S)$.
Recall that $\cC$ is a chamber complex on which $W$ acts simply-transitively chamber-wise (see \cite[Ch.~V]{Bourbaki2007}). Hence $\Cay(W,S)$ can (and will) be identified with the set of chambers of $\cC$, with two distinct chambers being adjacent in $\Cay(W,S)$ if they share a wall. 
The chamber complex $\cC$ can be realised geometrically as the complex determined by the walls of $(W,S)$ in the Tits cone, and we will identify both $\cC$ and $\cP_{W,M}$ with their geometric realisations inside the Tits cone.

The sets of walls and vertices of $\cC$ can be partitioned into types: a wall is of type $t$ if its reflection is conjugate to the element $t \in S$, while a vertex $v$ is of type $t$ if it lies on no wall of type $t$. A chamber in $\cC$ is delimited by one wall of each type, and contains one vertex of each type. The set of chambers containing a given vertex $v$ of type $t$ is in a bijective correspondence with $W_{t} = \langle S \setminus \{t\} \rangle$. 
More generally, the set of chambers of $\cC$ containing a given vertex $v$ of $\cP_{W,M}$ is in bijective correspondence with the stabilizer $W_v$ of $v$ in $W$. Note that $W_v$ is the conjugate of the standard parabolic subgroup $W_M = \langle S \setminus M \rangle$ by any element of $W$ which brings the standard vertex $x_0$ of $\cP_{W,M}$ to $v$. 

When $M=\{s\}$ consists of a single element (in particular, when $\cP_{W,M}$ is the universal polytope of a string Coxeter system), the vertices of $\cP_{W,M}$ are identified with the vertices in $\cC$ of type $s$ by construction. 
If $|M| \geq 2$, the vertices of $\cP_{W,M}$ are not vertices of $\cC$, because they do not lie on any hyperplane whose type belongs to $M$. 
Regardless of the size of $M$, two distinct vertices $v$ and $v'$ of $\cP_{W,M}$ are connected by an edge if and only if there are two neighbouring chambers of $\cC$ containing $v$ and $v'$ respectively. 

\begin{lemma} \label{lem:polytope-quasi-isometric}
Let $(W,S)$ be Coxeter system and $M$ a subset of $S$. The $1$-skeleton $X$ of the associated Wythoffian polytope $\cP_{W,M}$ and the Cayley graph $\Cay(W,S)$ are quasi-isometric if and only if $\cP_{W,M}$ has finite vertex links. In this case, the natural $W$-equivariant surjection $f: \Cay(W,S) \to X$ that sends a chamber to the unique vertex of $\cP_{W,M}$ it contains is a nonexpansive quasi-isometry.
\end{lemma}
\begin{proof}
Let $d_S$ (resp.~$d_X$) denote the geodesic distance in the graph $\Cay(W,S)$ (resp.~$X$). 
The map $f$ is clearly $W$-equivariant and surjective. The preimage under $f$ of a vertex $v$ of $\cP_{W,M}$ is the set of chambers in $\cC$ containing $v$, which forms a convex chamber subcomplex of $\cC$ (the \emph{link of $v$ in $\cC$}). 

Suppose that $\cP_{W,M}$ has finite vertex links and pick a vertex $v \in \cP_{W,M}$. 
If $C$ is chamber of $\cC$ containing $v$, then $v$ is connected by an edge of $\cP_{W,M}$ to its reflection $v'$ through any wall of $C$ not containing $v$. 
In consequence, the set $f^{-1}(v)$ of chambers containing $v$ must be a finite convex chamber subcomplex of $\cC$. 
Let $D$ denote its diameter (measured with $d_S$); then $D$ coincides with the length of the longest word in the finite Coxeter group $W_v \cong \langle S \setminus M \rangle$ (with respect to $S \setminus M$). 
If $\gamma$ is a geodesic in $\Cay(W,S)$, then $f(\gamma)$ is a walk in $X$ (possibly with repetitions). Hence for any $w, w' \in W$, we know that
\[
d_X(f(w),f(w')) \leq d_S(w,w').
\]
On the other hand, let $(v_0, \dots, v_n)$ be a geodesic in $X$ connecting $v_0 = f(w)$ to $v_n = f(w')$. Let $C_0$ and $C'_{n}$ be the chambers of $\cC$ corresponding to $w$ and $w'$ respectively, and let $C'_i$ and $C_{i+1}$ denote the adjacent chambers of $\cC$ that contain $v_i$ and $v_{i+1}$ respectively. 
Then for each $0 \leq i \leq n$, there is a path $\gamma_i$ of length at most $D$ connecting $C_i$ to $C'_i$ in the link (in $\cC$) of $v_i$. Concatenating the paths $\gamma_0, \dots, \gamma_n$, we see that 
\[
d_S(w,w') \leq (D+1) \cdot d_X(f(w),f(w')) + D,
\]
which proves the first implication. 

For the converse, it suffices to note that if $\cP_{W,M}$ has infinite vertex links, then the neighbourhood of a vertex in $X$ is infinite, while on the other hand, balls of finite radius in $\Cay(W,S)$ are finite.
\end{proof}

\begin{remark} \label{rem:complex-quasi-isometric}
Using similar arguments, one can show that the $1$-skeleton of the Coxeter complex $\cC$ of a (finitely generated) Coxeter system $(W,S)$ is quasi-isometric to $\Cay(W,S)$ if and only if $(W,S)$ is \emph{barely infinite}, that is, every proper parabolic subgroup of $W$ is finite. Unfortunately, the $1$-skeleton of a Coxeter complex is seldom a regular graph. 
\end{remark}

\subsection{Comparing quotients} \label{subsec:comparing-quotients}
For the remainder of this section, $(W,S)$ will be a Coxeter system and $M$ a subset of $S$ such that the associated Wythoffian polytope $\cP_{W,M}$ has finite vertex links. As before, $X$ denotes the $1$-skeleton of $\cP_{W,M}$. Let $N$ be a normal subgroup of $W$ and $\pi_N$ denote the quotient map $W \to W/N$. 

Note that the quotient graph $\Cay(W,S) / N$ is naturally isomorphic to the Cayley graph $\Cay(\pi_N(W), \pi_N(S))$. Hence the quasi-isometry $f$ arising from Lemma \ref{lem:polytope-quasi-isometric} induces a mapping $f_{N}$ defined by the following commuting diagram of $W$-equivariant surjections. 
\[
\begin{tikzcd}
{\Cay(W,S)} \arrow[r, "f"] \arrow[d, "\pi_N"'] & X \arrow[d]    \\
{\Cay(\pi_N(W), \pi_N(S))} \arrow[r, "f_N"']   & X / N
\end{tikzcd}
\]
Since geodesics in $\Cay(\pi_N(W), \pi_N(S))$ and $X / N$ can be lifted to geodesics in $\Cay(W,S)$ and $X$ respectively, the proof of Lemma \ref{lem:polytope-quasi-isometric} shows that $f_N$ is a quasi-isometry with the same quasi-isometry constants as $f$ (and in particular, these do not depend on $N$).

\subsection{Comparing regularity} \label{subsec:comparing-regularity}
In order to ensure that the graph $X / N$ retains the regularity of $X$, it suffices that the quotient map $X \to X / N$ is injective on the neighbourhood of any vertex of $X$ and creates no new triangles. 
(Note that the graph $X / N$ is always regular: $W / N$ acts transitively on $X / N$ because $W$ does so on the vertices of $\cP_{W,M}$. In fact, $X / N$ is even arc-transitive when $|M|=1$, since $W$ acts transitively on pairs of adjacent vertices of $\cP_{W,M}$ in this case. 
What is at stake here is the higher regularity, stemming for example from Lemma \ref{lem:higher-regularity}.)
In turn, because $N$ acts on $X$ by graph automorphisms, this can be achieved by requiring that the action of $N$ on $X$ has minimal displacement at least $4$. 
In view of Lemma \ref{lem:polytope-quasi-isometric}, this would follow if the action of $N$ on $\Cay(W,S)$ had minimal displacement at least $4(D+1)+D$, or in other words, if every nontrivial element in $N$ had length at least $5D+4$. 
The elements in $W$ whose lengths are less than $5D+4$ form a finite set $T$.

Because $W$ is a finitely generated linear group, it is residually finite by a classical theorem of Malcev \cite{Malcev1940}. Accordingly, let $\{N_m\}_{m \in I}$ be a collection of finite-index normal subgroups of $W$ which is closed under intersection and satisfies $\bigcap_{m \in I} N_m = \{1\}$. 
Let $I' = \{m \in I \mid T \cap N_m = \{1\}\}$, so that $\{N_m\}_{m \in I'}$ is again closed under intersection and satisfies $\bigcap_{m \in I'} N_m = \{1\}$. By the previous paragraph, for $m \in I'$ the graph $X / N_m$ has the same regularity as $X$. 
Note that if $W$ is infinite then the indices of the subgroups $N_m$ are necessarily unbounded, because finitely generated groups only have finitely many subgroups of a given finite index. 
\medskip

In summary, we have so far proved the following. 

\begin{proposition}
Let $(W,S)$ be a non-definite Coxeter system and $M$ a subset of $S$, such that the Wythoffian polytope $\cP_{W,M}$ has finite vertex links. Let $X$ be the $1$-skeleton of $\cP_{W,M}$ and let $(a_0, \dots, a_n)$ be the regularity of $X$ (in the sense of \S \ref{subsec:higher-regularity}). 
Given any collection $\{N_m\}_{m \in I}$ of finite-index normal subgroups of $W$ closed under intersection and satisfying $\bigcap_{m \in I} N_m = \{1\}$, there exists $I' \subset I$ with the same properties, such that for $m \in I'$ the quotient graphs $X / N_m$ are $(a_0, \dots, a_n)$-regular and have unbounded sizes. 
\end{proposition}

\subsection{Comparing expansion}\label{subsec:comp-exp}
It remains to determine when the collection $\{X / N_m\}_{m \in I}$ forms a family of expanders. Let $\pi_m$ and $f_m$ denote the maps constructed in \S\ref{subsec:comparing-quotients} for the subgroup $N = N_m$. 
The following well-known proposition, applied to $f_m$, indicates that to this end it is equivalent to investigate when the Cayley graphs $\Cay(\pi_m(W), \pi_m(S))$ form a family of expanders. 
This will be the subject of \S\ref{sec:density-superapproximation}. 

\medskip

Recall that a map $f$ from a metric space $(X,d_X)$ to a metric space $(Y,d_Y)$ is called a {\em quasi-isometry} if there exist constants $A\geq 1,B\geq 0,C\geq 0$ such that the following two conditions hold:
\begin{enumerate}[label=(\roman*)]
    \item $\forall \: x,x' \in X: A^{-1} d_X(x,x') - A^{-1}B \leq d_Y(f(x),f(x') \leq Ad_X(x,x')+B$,
    \item $\forall \: y \in Y: \exists x \in X: d_Y(y,f(x))\leq C$.
\end{enumerate}
Note that one can always weaken the conditions to $A=B=C=:D$; call this a {\em $D$-quasi-isometry}. 
When there is a quasi-isometry $f: X \to Y$, we call $(X,d_X)$ and $(Y,d_Y)$ \emph{quasi-isometric}. It is an easy exercise to show that when $f: X \to Y$ is a quasi-isometry, there exists a \emph{quasi-inverse} to $f$, that is, a quasi-isometry $g: Y \to X$ with constants depending only on those of $f$, and for which $f \circ g$ and $g \circ f$ have displacement bounded by the quasi-isometry constants of $f$. In consequence, quasi-isometry defines an equivalence relation between metric spaces. 

\begin{proposition} \label{lem:comparing-expansion}
Let $D \geq 1$ and let $f: Y \to Z$ be a $D$-quasi-isometry between two finite connected graphs $Y$ and $Z$. Then there exist constants $c$, $c' > 0$ depending only on the quasi-isometry constants of $f$ (or equivalently, on $D$) and on the maximum degrees of $Y$ and $Z$, such that if $h(Y) \geq \epsilon$, then $h(Z) \geq \min(c \epsilon, c')$. 
\end{proposition}
\begin{proof}
Let $a_Y$ and $a_Z$ be the maximum degrees of $Y$ and $Z$. Let $c_1,c_3 \geq 1$ and $c_2,c_4 \geq 0$ be such that
\[
c_1^{-1} d_Y(y,y') - c_1^{-1} c_2 \leq d_Z(f(y),f(y')) \leq c_3 d_Y(y,y') + c_4 \quad \textrm{for all $y,y' \in Y$},
\]
and let $c_5$ be such that every vertex of $Z$ lies at distance at most $c_5$ from $f(Y)$. 
Note that there exists $c_6$ (depending only on $c_2$ and $a_Y$) such that $|f^{-1}(z)| \leq c_6$ for any $z \in Z$. 

Set $r = (4 a_Z^{c_5})^{-1}$ (noting that $r\leq \frac{1}{4})$. Pick $V \subset Z$ such that $0 < |V| \leq \frac{1}{2}|Z|$. We distinguish three cases.
\smallskip

First, suppose $|f(Y) \setminus V| \leq r |f(Y)|$. Then no more than $a_Z^{c_5} |f(Y) \setminus V| \leq r a_Z^{c_5} |Z| = \frac{1}{4} |Z|$ vertices lie at distance at most $c_5$ from $f(Y) \setminus V$. 
The remaining $\frac{3}{4}|Z|$ or more vertices must then lie at distance at most $c_5$ from $f(Y) \cap V$, with at least $\frac{1}{4}|Z|$ of them lying outside $V$. 
At the same time, the paths of length $c_5$ leaving $V$ via a given edge of $\partial V$ cannot reach more than $a_Z^{c_5-1}$ vertices altogether. 
As a consequence, at least $(4a_Z^{c_5-1})^{-1} |Z|$ edges must be leaving $V$, which shows that $|\partial V| / |V| \geq (2a_Z^{c_5-1})^{-1}$.
\smallskip

Second, suppose $|f(Y) \cap V| \leq 2r|V|$. Then no more than $a_Z^{c_5} |f(Y) \cap V| \leq 2r a_Z^{c_5}|V| = \frac{1}{2}|V|$ vertices lie at distance at most $c_5$ from $f(Y) \cap V$, so the remaining $\frac{1}{2}|V|$ or more vertices of $V$ must lie at distance at most $c_5$ from $f(Y) \setminus V$. As above, this implies that at least $(2a_Z^{c_5-1})^{-1}|V|$ edges are leaving $V$, which again shows that $|\partial V| / |V| \geq (2a_Z^{c_5-1})^{-1}$.
\smallskip

Third and foremost, suppose $2r|V| \leq |f(Y) \cap V| \leq (1-r)|f(Y)|$. Then
\[
|f^{-1}(V)| \geq 2r|V|
\quad \textrm{and} \quad
|f^{-1}(Z \setminus V)| \geq |f(Y) \setminus V| \geq r|f(Y)| \geq \frac{2r^2}{1-r}|V|.
\]
As $h(Y) \geq \epsilon$, this implies that
\[
|\partial f^{-1}(V)| \geq \epsilon \min(|f^{-1}(V)|, |f^{-1}(Z \setminus V)|) \geq  c^{-1}_7 \epsilon |V|,
\]
where $c_7 = \frac{1-r}{2r^2}$. Hence there are at least $a_Y^{-1} c_7^{-1} \epsilon |V|$ vertices of $Y$ connected to but not lying in $f^{-1}(V)$. 
Their images under $f$ form a set of at least $(a_Y c_6 c_7)^{-1} \epsilon |V|$ vertices of $Z$, lying at distance at most $c_3 + c_4$ from $V$ outside of $V$. As a consequence, at least $(a_Y a_Z^{c_3+c_4-1} c_6 c_7)^{-1} \epsilon |V|$ edges are leaving $V$. 
\end{proof}

\begin{remark}
If $f$ happens to be surjective (as is the case in our setting), the proof of Proposition \ref{lem:comparing-expansion} simplifies considerably. The first two cases are irrelevant, and in the third, one easily obtains $h(Z) \geq (a_Y a_Z^{c_3+c_4-1} c_6)^{-1} \epsilon$ by using the fact that $|f^{-1}(Z \setminus V)| \geq |V|$.
\end{remark}

With the existence of quasi-inverses in mind, the following is an immediate corollary to Proposition \ref{lem:comparing-expansion}. 

\begin{corollary} \label{cor:comparing-expansion}
Let $\{Y_m\}_{m \in J}$ and $\{Z_m\}_{m \in J}$ be two families of graphs of bounded maximum degree, indexed by a set $J$. Suppose that there is a $D$-quasi-isometry $f_m: Y_m \to Z_m$ for every $m \in J$. Then $\{Y_m\}_{m \in J}$ is a family of expanders if and only if $\{Z_m\}_{m \in J}$ is. 
\end{corollary}

\section{Superapproximation for indefinite Coxeter groups} \label{sec:density-superapproximation}

Since $S$ is assumed to be finite, $W$ is a discrete subgroup of $\Orth_B(\R)$ (see \S\ref{subsec:Coxeter-geometric-representation}). This implies that if $(W,S)$ is semidefinite (resp.~definite), then $W$ is virtually abelian (resp.~finite) \cite[Ch.~V.4]{Bourbaki2007}. 
As virtually abelian groups are amenable, there is no chance to witness superapproximation or expansion phenomena in $(W,S)$ if it is semidefinite. 

The situation is quite different for indefinite Coxeter groups, as attested by the following theorem of Benoist and de la Harpe. 

\begin{theorem}[{\cite[Th{\'e}or{\`e}me]{BenoistdelaHarpe2004}}] \label{thm:CoxeterDensity}
Let $(W,S)$ be an indefinite irreducible Coxeter system, and let $\rad(B)$ denote the radical of the associated bilinear form $B$. Then the Zariski-closure of $W$ in $\Orth_B$ is precisely the kernel $\Orth_B^1$ of the restriction map $\Orth_B \to \GL_{\rad(B)}: g \mapsto g_{|_{\rad(B)}}$. 
In particular, if $B$ is non-degenerate, then $W$ is Zariski-dense in $\Orth_B$. 
\end{theorem}

A consequence of Theorem \ref{thm:CoxeterDensity} is that the connected component $\Orth_B^{1 \circ}$ of the Zariski-closure of the indefinite Coxeter group $W$ is perfect. Indeed, if $B'$ denotes the bilinear form induced by $B$ on $V' = V / \rad(B)$, then $\Orth_B^{1 \circ} \cong \SO_{B'} \ltimes V'^{\dim \rad(B)}$, with the latter being a perfect group because $\SO_{B'}$ is simple and $V'$ is an irreducible $\SO_{B'}$-module. 
This is precisely the ingredient needed for us to apply the following \emph{superapproximation} theorem due to Salehi Golsefidy. We will use it to deduce that congruence quotients of the Cayley graph of an indefinite Coxeter group form a family of expanders. 

\medskip

Fix non-zero integers $N_0$ and $q_0$. For any integer $m$ coprime to $q_0$, let $\pi_m$ denote the quotient map $\GL_{N_0}(\Z[1/q_0]) \to \GL_{N_0}(\Z / m \Z)$ induced by reduction modulo $m$. 

\begin{theorem}[{\cite[Theorem 1]{SalehiGolsefidy2019}}] \label{thm:Superapproximation}
Let $\Gamma$ be the group generated by a finite symmetric subset $S$ of $\GL_{N_0}(\Z[1/q_0])$. Suppose that $\Gamma$ is infinite. Fix $M_0 \in \N$. The family of Cayley graphs $\{\Cay(\pi_m(\Gamma), \pi_m(S))\}_{m}$, as $m$ runs through either
\[
\{p^n \mid n \in \N, p \textrm{ prime}, p \nmid q_0 \}
\quad \textrm{or} \quad
\{m \in \N \mid \gcd(m,q_0) = 1, p^{M_0+1} \nmid m \textrm{ for any prime } p\},
\]
is a family of expanders if and only if the connected component $G^\circ$ of the Zariski-closure $G$ of $\Gamma$ in $\GL_{N_0}$ is perfect. 
\end{theorem}

In order to apply Theorem \ref{thm:Superapproximation} to an indefinite Coxeter group $W$, it remains for us to observe that $W$ can indeed be seen as a subgroup of $\GL_{N_0}(\Z[1/q_0])$. The attentive reader may foresee Weil's trick of restricting scalars. 
The entries of the matrix of $2B$ in the canonical basis of $V$ are algebraic integers, and so there exists a number field $K$, with ring of integers $\cO_K$, over which the algebraic group $\Orth_B$ can be defined in such a way $W \subset \Orth_B(\cO_K)$. The restriction of scalars $\Res_{K/\Q}(\Orth_B)$ is a linear algebraic $\Q$-group, and as such can be embedded over $\Q$ in $\GL_{N_0}$ for some $N_0$. If one is careful with the construction of $\Res_{K/\Q}(\Orth_B)$ and the choice of the embedding in $\GL_{N_0}$, then one can ensure that the image of $W$ lies in $\GL_{N_0}(\Z)$. Otherwise, let $q_0$ be a lowest common denominator of the entries of the image of $S$. Then $S$, and hence also $W$, is a subset of $\GL_{N_0}(\Z[1/q_0])$. The Zariski-closure of $W$ in $\GL_{N_0}$ is the image of $\Res_{K/\Q}(\Orth_B^{1\circ})$, which is perfect since $\Orth_B^{1\circ}$ is perfect. 
\medskip

On the way, we also record the following very useful corollary to Theorem \ref{thm:Superapproximation}, which would already be sufficient to construct expanders from an indefinite Coxeter group. 

\begin{corollary}\label{cor:superapproximation}
Let $\Gamma$ be a linear group (in characteristic 0) generated by a finite symmetric set $S$. Suppose that $\Gamma$ is not virtually solvable. Then there exists a collection $\{N_m\}_{m \in I}$ of normal subgroups of $\Gamma$ whose indices are unbounded, and for which the Cayley graphs $\Cay(\pi_m(\Gamma), \pi_m(S))$ form a family of expanders, where $\pi_m$ denotes the quotient map $\Gamma \to \Gamma / N_m$. 
\end{corollary}
\begin{proof}
Let $F$ be a field of characteristic 0 such that $\Gamma$ is a subgroup of $\GL_{N}(F)$, and let $A$ be the $\Q$-subalgebra of $F$ generated by the entries of the elements of $\Gamma$. By assumption, $A$ is a finitely generated $\Q$-algebra. 
There exists a morphism $A \to \overline{\Q}$ inducing a map $\varphi: \GL_{N}(A) \to \GL_{N}(\overline{\Q})$, for which the image $\varphi(\Gamma)$ of $\Gamma$ in $\GL_{N}(\overline{\Q})$ is still not virtually solvable \cite[Proposition~2.2]{LubotzkyMann1991}. Since $\Gamma$ is finitely generated, $\varphi(\Gamma)$ lies in some number field $K$. 
After restricting scalars from $K$ down to $\Q$ if necessary, we may assume that $\varphi(\Gamma)$ actually lies in $\GL_{N'}(\Q)$. 

Let $G$ be the Zariski-closure of $\varphi(\Gamma)$ in $\GL_{N'}$. Let $R$ denote the solvable radical of $G$ and $\psi: G \to G / R$ the quotient map onto the $\Q$-group $G/R$. By construction, the connected component $H$ of $G/R$ is a semisimple $\Q$-group. 
If $H$ were trivial, $G$ would be a finite extension of the solvable group $R$, hence $\varphi(\Gamma)$ would be virtually solvable. We deduce that $\Gamma' = \psi(\varphi(\Gamma))$ is Zariski-dense in the $\Q$-group $G/R$, whose nontrivial connected component is semisimple hence perfect. 

Now embed $G/R$ into $\GL_{N''}$ over $\Q$ for some $N''$, and apply Theorem \ref{thm:Superapproximation} to obtain a collection of congruence subgroups $\{N_m\}_{m \in I}$ of $\Gamma'$ whose indices in $\Gamma'$ are unbounded, and for which the Cayley graphs of the quotients $\Gamma' / N_m$ are expanders. The preimage in $\Gamma$ of $\{N_m\}_{m \in I}$ then verifies the statement of the corollary. 
\end{proof}
\medskip

We now have all the pieces to prove Theorem \ref{thm:Main}.

\subsection{Proof of Theorem \ref{thm:Main}} \label{subsec:proof-Main}
Let $(W,S)$, $\cP_{W,M}$ and $X$ be as in the statement of the theorem. 

Applying Theorem \ref{thm:Superapproximation} and the surrounding discussion to $(W,S)$, we find a collection of congruence subgroups $N_m = \ker \pi_m$ of $W$ (with respect to the restriction of scalars of the geometric representation), for which $\{\Cay(\pi_m(W),\pi_m(S))\}_m$ forms a family of expanders, as $m$ runs through
\[
I = \{p^n \mid n \in \N, \ p \in \N \textrm{ prime}\}
\ \cup \ 
\{m \in \N \mid p^{M_0+1} \nmid m \textrm{ for any prime } p \in \N\}. 
\]

As shown in \S\ref{subsec:comparing-quotients}, there are quasi-isometries $f_m: \Cay(\pi_m(W),\pi_m(S)) \to X / N_m$ with constants depending only on $(W,S)$. By Corollary \ref{cor:comparing-expansion}, the graphs $\{X / N_m\}_{m \in I}$ form a family of expanders. 

We can refine the choice of the subset $I'$ in \S\ref{subsec:comparing-regularity} as follows. 
The finite set $T$ from \S\ref{subsec:comparing-regularity} intersects $N_p$ nontrivially only for finitely many primes $p$, and for each of those, $T$ intersects $N_{p^n}$ nontrivially only for finitely many $n$. It follows that the set $\{m \in I \mid T \cap N_m = \{1\}\}$ is cofinite in $I$, and in view of \S\ref{subsec:comparing-regularity}, a fortiori so is the set
\[
I' = \{m \in I \mid X / N_m \textrm{ has the same regularity as } X\}.
\]

Altogether, the graphs $\{X / N_m\}_{m \in I'}$ are $(a_0, \dots, a_n)$-regular, and form an infinite family of expanders. 
This concludes the proof.

\subsection{Proof of Corollary \ref{cor:Main-string}} \label{subsec:proof-cor-Main-string}

As already mentioned in \S \ref{sec:introduction}, to deduce Corollary \ref{cor:Main-string} from Theorem \ref{thm:Main}, it suffices to first apply Lemma \ref{lem:higher-regularity} to the universal polytope of the string Coxeter system, which is a regular polytope. 

\begin{remark} \label{rem:expanders-nonexpanders} 
It should be stressed that quotients of indefinite Coxeter groups give rise to both expanders and non-expanders. Most families of subgroups do not give expanders; only careful choice, as with the congruence subgroups used in the proof of Theorem \ref{thm:Main}, leads to expanders. 
More precisely, Coxeter groups which are lattices in $\Orth_{n,1}$ for some $n \in \NN$ (for example $\HH_5$, or any of the groups from Table \ref{tab:regular-hyperbolic-honeycombs}) have finite-index subgroups which map onto a nonabelian free group (see \cite[Corollary 3.6]{Lubotzky1996}). Noskov and Vinberg \cite{NoskovVinberg2002} showed that this even holds for all finitely generated subgroups of general Coxeter groups, provided they are not virtually abelian (in particular, this holds for indefinite Coxeter groups). Such finite-index subgroups have infinite residually finite amenable quotients which lead to non-expanding finite quotients. 
\end{remark}

\section{High-dimensional expanders} \label{sec:HDX}

As mentioned in \S \ref{subsec:higher-regularity}, it is natural to think about a $(a_0,\dots,a_{n-1})$-regular graph $X$ as an $n$-dimensional simplicial complex, more precisely as the $n$-skeleton $X^{(n)}$ of the clique complex $X^\textrm{cl}$ of $X$, of which $X$ is the 1-skeleton. 
In the case when $X$ is an expander graph, it is natural to wonder whether $X^{(n)}$ has high-dimensional expansion properties. 

The notion and study of high-dimensional expanders (HDX) has been very popular in recent years with many different definitions  which are not equivalent to each other in general (see \cite{Lubotzky-2018} and the references therein). 
There is a priori no reason to believe that if $X$ is an expander graph, then $X^{(n)}$ is an HDX in any of the definitions. 
Nevertheless, we observe that Proposition \ref{lem:comparing-expansion} yields some ``high-dimensional information''. In \cite{Kaufman-Mass-2017} and \cite{Dinur-Kaufman-2017}, a systematic study of $i$-walks on a simplicial complex $Y$ was initiated. Here an $i$-walk means a walk on the $i$-cells, where $i$-cells are adjacent if they are contained in a common $(i+1)$-cell. The following basic question was asked in \cite{Kaufman-Mass-2017}: 
`Are there bounded degree high-dimensional simplicial complexes in which all the high order random walks converge rapidly to their stationary distribution?'
Proposition \ref{lem:comparing-expansion} gives us a quick answer to this. 

\begin{corollary}\label{cor:HDX}
Let $Y$ be a bounded-degree connected simplicial complex of dimension $d$. Let $\overline{Y(i)}$ be the graph whose vertices are the $i$-cells of $Y$ and where two $i$-cells are adjacent if they are contained in a common $(i+1)$-cell. 
Then for every $0 \leq i < d$ such that $\overline{Y(i)}$ is connected, $\overline{Y(i)}$ is an expander graph if and only if $\overline{Y(0)}$ (that is, the $1$-skeleton of $Y$) is an expander graph. 
\end{corollary}
\begin{proof}
When $Y$ is of bounded degree and $\overline{Y(i)}$ is connected, $\overline{Y(i)}$ is quasi-isometric to $\overline{Y(0)}$ (with constants depending only on the degree bound), and the statement follows from Proposition \ref{lem:comparing-expansion}.
\end{proof}

Recall that the (lazy) random walk on an expander graph converges rapidly to the stationary distribution. So formally speaking, Corollary \ref{cor:HDX} answers the above question, since there are various ways to construct complexes $Y$ satisfying its hypotheses (and indeed for some of them \cite{Kaufman-Mass-2017} proved this).
However, experience with HDX shows that one needs quantitative results (of the type given in \cite{Kaufman-Mass-2017,Dinur-Kaufman-2017}) for concrete applications. 

\medskip

Finally, let us mention another interesting remark related to high-dimensional expanders. As was explained in \S\ref{sec:density-superapproximation}, our expander graphs are obtained by applying super-approximation to normal congruence subgroups of the Coxeter system $(W,S)$. 
The Coxeter systems studied here are all known to have finite index subgroups which map onto non-abelian free groups \cite{NoskovVinberg2002}. This implies that at the same time, we could choose suitable (non congruence) normal subgroups which give rise to HR-graphs with exactly the same local structure but which are not expanders.

This is of interest in light of Garland theory \cite{Garland-1973}. Garland theory shows that the $1$-skeleton of simplicial complexes are expanders when the links are ``sufficiently good'' expanders. 
Our examples show that in fact the links need to be strong enough expanders in order to deduce global expansion. 
Thus the popular statement saying that Garland theory implies that expansion of high-dimensional simplicial complexes is a local property should be formulated in a careful and quantitative way.

\section{Highly regular honeycombs in hyperbolic space} \label{sec:highly-regular-polytopes}

In this section, we illustrate how Theorem \ref{thm:Main} can be used to produce examples of highly regular expander graphs. In particular, these examples will prove Theorem \ref{thm:levels}. 

As we have seen in \S\ref{sec:comparing-graphs} and \S\ref{sec:density-superapproximation}, in order to obtain expander graphs from the 1-skeleton of a Wythoffian polytope $\cP_{W,M}$ with finite vertex links, the Coxeter system $(W,S)$ should be indefinite. In particular, this happens when $\cP_{W,M}$ is a tessellation of hyperbolic space. 
At the same time, to obtain a graph of regularity level $n$, it would suffice that $\cP_{W,M}$ has all faces of rank $n+1$ simplicial, and that the 1-skeleton of the links of faces of rank $\leq n$ are transitive graphs. 
When $\cP_{W,M}$ is a regular polytope, it suffices that a single face of rank $n+1$ is simplicial (as the second condition is superseded by Lemma \ref{lem:higher-regularity}). 

\subsection{Regular tessellations of hyperbolic space} \label{subsec:regular-hyperbolic-tesselations}
The hyperbolic plane can be tessellated by regular triangles of angles $2\pi / a$, for every integer $a > 6$. The stabiliser in $\PGL_2(\R)$ of such a tiling $\cT_a$ is a cocompact lattice, the so-called \emph{$(2,3,a)$-triangle group}, here denoted by $D_a$. 
The quotients of $D_a$ (or equivalently, of $\cT_a$) give rise to infinitely many $(a,2)$-connected regular graphs, as pointed out in \cite[Example 1.2(5) \& Section 6]{ChapmanLinialPeled2019}. 
But Corollary \ref{cor:Main-string} says further that given $a > 6$, infinitely many of these $(a,2)$-regular graphs (namely the ones obtained from congruence subgroups in the proof of the theorem) form a family of expanders. 
(It is worth pointing out that only for finitely many values of $a$ the triangle group $D_a$ is an arithmetic lattice; see \cite{Takeuchi1977}. So the meaning of `congruence subgroups' is as given by the proof.)
Moreover, for each $a > 6$, the group $D_a$ has a finite index subgroup which is the fundamental group of a Riemann surface and hence maps onto a nonabelian free group, and therefore has many quotients which do not yield expanders (cf.~Remark \ref{rem:expanders-nonexpanders}). 

Sadly, tessellations of hyperbolic space by regular (possibly ideal) polytopes are scarce in dimensions 3 and above. 
In contrast with the hyperbolic plane, there are only finitely many regular hyperbolic tessellations in dimensions 3, 4 and 5, and none in dimensions 6 and above. 
This fact was already know to Schlegel \cite{Schlegel1883}, who initiated their study. The full list can be found in \cite{Coxeter1956}. 
A quick look through the tables suggests the following candidates, the regularity of which can be computed using Lemma \ref{lem:higher-regularity}. 

\subsection{Table} \label{tab:regular-hyperbolic-honeycombs}
Some regular hyperbolic tessellations. 
\begin{center}
\begin{tabular}{@{}llll@{}}
\hline
    Regular polytope $\cP$ & Schl{\" a}fli symbol 
& Regularity of $X$ \\ \hline
    Icosahedral honeycomb & $\{3,5,3\}$ 
& $(20,3,0)$\\
    Order-5 4-simplicial honeycomb & $\{3,3,3,5\}$ 
& $(120,12,5,2,0)$ \\
    Pentagrammic-order hexacosichoric honeycomb & $\{3,3,5,5/2\}$ 
& $(120,12,5,0)$ \\
    Order-5 icosahedral hecatonicosachoric honeycomb & $\{3,5,5/2,5\}$
& $(120,12,0)$ \\
    Order-3 5-orthoplicial honeycomb & $\{3,3,3,4,3\}$ 
& $(\aleph_0,24,8,3,0)$ \\ 
    Order-3 4-orthoplicial honeycomb honeycomb & $\{3,3,4,3,3\}$ 
& $(\aleph_0,16,4,0)$ \\ 
    Order-3 icositetrachoric honeycomb honeycomb & $\{3,4,3,3,3\}$
& $(32,5,0)$ \\
\hline
\end{tabular}
\end{center}

\begin{remark} \label{rem:limited-regular-examples}
The third example is a faceting of the second one. Hence there is essentially one example of hyperbolic tessellation whose $1$-skeleton is $(a_0,a_1,a_2)$-regular with $a_0 \in \N$ and $a_2 \neq 0$, namely the one with Schl{\" a}fli symbol $\{3,3,3,5\}$ and Coxeter diagram 
\begin{tikzpicture}[scale=0.7, transform shape]
\tikzstyle{every node}=[draw,solid,draw=black,fill=black,shape=circle,minimum size=.05cm,inner sep=.05cm]
\path[solid,draw=black,fill=white,thin] 
(0,0) node[double] (root 1) {} --++(.5,0) node (root 2) {} --++(.5,0) node (root 3) {} --++(.5,0) node (root 4) {} -- node[draw=none,fill=none,shape=circle,above]{5} ++(.5,0) node (root 5) {};
\end{tikzpicture}. 

In fact, even among arbitrary indefinite string Coxeter systems, this is the only example which can yield an infinite family of $(a_0,a_1,a_2)$-regular quotient graphs (with $a_0 \in \N$ and $a_2 \neq 0$). 
Indeed to achieve this, the diagram of the Coxeter system $(W,S)$ should start with at least two consecutive edges labelled 3, while the stabiliser $W_0$ of a given vertex of $\cP_W$ (whose Coxeter diagram is obtained by removing the 0\textsuperscript{th} vertex and its edge from the diagram of $(W,S)$) should be finite (see Remark \ref{rem:regularity-Schlafli-symbol} and \S\ref{subsec:Wythoffian-polytopes}). 
In other words, the Coxeter diagram of $(W,S)$ should be obtained by adding an edge labelled 3 to the string diagram of a finite Coxeter group $W_0$ which already starts with at least one edge labelled 3, in such a way that the resulting group is infinite. 
The only possible candidates for the finite group are 
\begin{tikzpicture}[scale=0.7, transform shape]
\tikzstyle{every node}=[draw,solid,draw=black,fill=black,shape=circle,minimum size=.05cm,inner sep=.05cm]
\path[solid,draw=black,fill=white,thin] 
(0,0) node (root 1) {} --++(.5,0) node (root 2) {} -- node[draw=none,fill=none,shape=circle,above]{4} ++(.5,0) node (root 3) {} --++(.5,0) node (root 4) {};
\end{tikzpicture} ($\rF_4$) 
and 
\begin{tikzpicture}[scale=0.7, transform shape]
\tikzstyle{every node}=[draw,solid,draw=black,fill=black,shape=circle,minimum size=.05cm,inner sep=.05cm]
\path[solid,draw=black,fill=white,thin] 
(0,0) node (root 1) {} --++(.5,0) node (root 2) {} --++(.5,0) node (root 3) {} -- node[draw=none,fill=none,shape=circle,above]{5} ++(.5,0) node (root 4) {};
\end{tikzpicture} ($\rH_4$). 
Unfortunately, extending the former diagram yields 
\begin{tikzpicture}[scale=0.7, transform shape]
\tikzstyle{every node}=[draw,solid,draw=black,fill=black,shape=circle,minimum size=.05cm,inner sep=.05cm]
\path[solid,draw=black,fill=white,thin] 
(0,0) node (root 1) {} --++(.5,0) node (root 2) {} --++(.5,0) node (root 3) {} -- node[draw=none,fill=none,shape=circle,above]{4} ++(.5,0) node (root 4) {} --++(.5,0) node (root 5) {};
\end{tikzpicture} ($\tilde{\rF}_4$), 
whose Coxeter system is semidefinite: it is the affine Weyl group of type $\tilde{\rF}_4$. 

Aside from the triangular tilings of the hyperbolic plane already mentioned, the only regular hyperbolic tessellations to yield $(a_0,a_1)$-regular graphs (with $a_0 \in \N$, $a_1 \neq 0$) are included in the table above. 
\end{remark}

\subsection{The order-5 4-simplex honeycomb} \label{subsec:order-5-4-simplex-honeycomb}
As an illustration, we work out the most noteworthy regular example. In this subsection, let thus $\cP$ denote the order-5 4-simplex honeycomb, the automorphism group of which is the Coxeter group $W$ with diagram 
\begin{tikzpicture}[scale=0.7, transform shape]
\tikzstyle{every node}=[draw,solid,draw=black,fill=black,shape=circle,minimum size=.05cm,inner sep=.05cm]
\path[solid,draw=black,fill=white,thin] 
(0,0) node (root 1) {} --++(.5,0) node (root 2) {} --++(.5,0) node (root 3) {} --++(.5,0) node (root 4) {} -- node[draw=none,fill=none,shape=circle,above]{5} ++(.5,0) node (root 5) {};
\end{tikzpicture}, 
sometimes known as $\rH_5$. This example already answers the question asked in \cite{ChapmanLinialPeled2019} positively.

Let $\varphi = \frac{1+\sqrt{5}}{2} \in \R$ and let $K = \Q(\varphi)$. The bilinear form $B$ on $\R^5$ associated with $W$ has matrix
\[
\frac{1}{2}
\begin{smallpmatrix}
2 & -1 & 0 & 0 & 0 \\
-1 & 2 & -1 & 0 & 0 \\
0 & -1 & 2 & -1 & 0 \\
0 & 0 & -1 & 2 & -\varphi \\
0 & 0 & 0 & -\varphi & 2
\end{smallpmatrix}
\]
with respect to the canonical basis $\{e_0, \dots, e_4\}$. It is an easy exercise to see that $B$ is equivalent over $K$ to the diagonal form $B' = \langle 1, 1, 1, 1, -\varphi \rangle$. It follows that $\Orth_B \cong \Orth_{B'}$ as algebraic $K$-groups, and $W$ has signature $(+^4, -)$. 

The 2-sheeted hyperbola $\{v \in \R^5 \mid B(v,v) = -1\}$ is preserved by $\Orth_B$, and each of the two sheets $\cH$ and $\cH^-$ is preserved by the group $W$. The space $\cH$ (or $\cH^-$ for that matter) is the Minkowski model for hyperbolic 4-space; its isometry group is $\Orth_B^+(\R) = \{g \in \Orth_B(\R) \mid g \cH = \cH\} \xrightarrow{\raisebox{-0.7ex}[0ex][0ex]{$\sim$}} \PO_B(\R)$. 

Next, let $\Orth_B(\cO_K)$, $\Orth_B^+(\cO_K)$ and $\SO_B(\cO_K)$ respectively denote the matrices in $\Orth_B(K)$, $\Orth_B^+(K)$ and $\SO_B(K)$, with entries in the ring of integers $\cO_K$ of $K$. By construction of the geometric representation (see \S\ref{subsec:Coxeter-geometric-representation}), the images of the generators $\{s_0, \dots, s_4\}$ of $W$ lie in $\Orth_B^+(\cO_K)$.
Each generator acts on $\cH$ as a hyperbolic reflection. The hyperplane arrangement generated by these reflections tessellates $\cH$ by compact 4-simplices, and this tessellation is a geometric representation of the Coxeter complex of $W$. 

The order-5 4-simplex honeycomb $\cP$ can be recovered by regrouping the tiles around each vertex of type $4$. Alternatively, $\cP$ can be obtained by playing kaleidoscope with a point placed on the hyperplanes associated with $s_1, \dots, s_4$ but not $s_0$, equidistantly from the surrounding hyperplanes. 
The link $L$ of a vertex of $\cP$ is a hexacosichoron (600-cell), which has 120 vertices, 720 edges and 1200 faces. At each vertex of $L$, 12 edges meet, with 5 faces around each of those edges and 2 cells containing each such face. The link of an edge of $\cP$ is an icosahedron. 

It follows from these geometric observations that $W$ is a cocompact lattice in $\Orth_B(\R)$. In this case, the conclusion of Theorem \ref{thm:CoxeterDensity} can also be obtained directly by applying Borel's density theorem \cite{Borel1960}: $W$ is Zariski-dense in $\Orth_B$. 
It also follows that $W$ has finite index in $\Orth_B(\cO_K)$. Indeed, by a classical theorem of Borel and Harish-Chandra, $\Orth_B(\cO_K)$ is also a lattice in the group $\Orth_B(\R)$. 
Alternatively, since $\Orth_B(\cO_K)$ is a discrete subgroup of $\Orth_B(\R)$ containing the lattice $W$, it must be lattice.

One can now apply Corollary \ref{cor:Main-string} (see also \S \ref{subsec:proof-Main}) to construct the finite $(120,12,5,2)$-regular congruence quotients of the 1-skeleton $X$ of $\cP$, which form an infinite family of expanders. 



\begin{remark}
In the example considered in \S\ref{subsec:order-5-4-simplex-honeycomb} (but also for other hyperbolic tilings), the subgroups $N_{m} = W \cap \ker \pi_{m}$ can be described explicitly via the geometric realisation. 
They act on $\cH$, and if $N_{m}$ is chosen appropriately (namely as in \S\ref{subsec:comparing-regularity}, with the additional requirement of being torsion-free),  
then the quotient $\cH / N_{m}$ is a compact hyperbolic 4-manifold which admits a triangulation (descending from the honeycomb) whose $1$-skeleton is precisely the graph $X / N_{m}$. 

Of course, the manifolds $\cH / N_{m}$ all cover the orbifold $\cH / W$. It is possible that in fact all but finitely many cover the manifold $\cH / \Sigma$ from \cite[\S5]{ConderMaclachlan2005}, which to this day achieves the smallest known volume among compact hyperbolic 4-manifolds. 
\end{remark}

\subsection{Wythoffian tessellations of hyperbolic space} \label{subsec:Wythoffian-hyperbolic-tesselations}
In addition to the regular ones, one can construct examples using more general {Wythoffian polytopes} (see \S \ref{subsec:Wythoffian-polytopes}). 
A strategy to obtain Wythoffian polytopes $\cP_{W,M}$ with a 1-skeleton $X$ of regularity level $n+1$, to which one can apply Theorem \ref{thm:Main}, goes as follows.

\smallskip

First, to ensure that all $(n+1)$-faces are simplices, the set $M$ should consist of a single vertex $s_0$ of the Coxeter diagram of $(W,S)$, and any connected subdiagram of size $n$ containing $s_0$ should be a path starting at $s_0$ containing only unlabelled edges. In consequence, $s_0$ is a leaf of the Coxeter diagram, and in the diagram there is a unique path starting at $s_0$ of length $n-1$. 

When all $(n+1)$-faces of $\cP_{W,M}$ are simplices, the $(n+1)$-skeleta of $\cP_{W,M}$ and of $X^\textrm{cl}$ coincide. In particular, if $X$ happens to have regularity level at least $i+1$, its $i$\textsuperscript{th} regularity degree $a_i$ equals the number of $(i+1)$-faces containing any $i$-face, or in other words, equals the size of the link of any $i$-face in $\cP_{W,M}$ ($i \leq n$). 

Now to check that the 1-skeleton $X$ of such a polytope $\cP_{W,M}$ indeed has regularity level $n$, we argue as follows. Since $s_0$ lies at the end of a unique path of length $n-1$ in the diagram, one sees inductively that for $-1 \leq i \leq n-1$, the link of an $i$-face is again a Wythoffian polytope, hence is vertex-transitive. This implies that the sphere $S_X(C)$ of radius 1 around any clique $C$ of size $i+1$ in $X$ is a vertex-transitive graph. In particular, it is $a_{i+1}$-regular, with $a_{i+1}$ again counting the number of $(i + 2)$-faces containing a given $(i+1)$-face.
This analysis also shows that the link of an $i$-face remains connected as long as $i \leq n-1$, so that $X$ is $(a_0, \dots, a_{n-1})$-connected regular. 

\smallskip

Second, the vertex links of $\cP_{W,M}$ are finite precisely when the subgroup $W_0= \langle S \setminus M \rangle$ of $W$ is finite, or in other words, when the Coxeter diagram obtained by removing $M$ and the edges containing it is that of a definite Coxeter system. 

It remains to determine when the system $(W,S)$ is indefinite. Since $(W_0, S \setminus M)$ is positive definite, this is the case precisely when the discriminant of the bilinear form $B$ of $(W,S)$ is negative (see \S \ref{subsec:Coxeter-geometric-representation}). 

\smallskip

In the following table, we record some examples of highly regular Wythoffian polytopes obtained via this strategy. In each case, the regularity degrees are computed using the sizes of the finite polytopes which appear as links of (simplicial) faces. 
Because at the last step the links split into a product, the connected regularity level of each example is one less than the regularity level indicated in the table. 
The vertex links are easily seen to be finite, and the sign of the discriminant of the bilinear form $B$ is checked by hand (or by computer).

\subsection{Table} \label{tab:Wythoffian-hyperbolic-honeycombs}
Some Wythoffian hyperbolic tessellations. 
\begin{center}
\begin{tabular}{@{}lclll@{}}
\hline
Coxeter group $W$ & Adorned diagram of $\cP_{W,M}$ & Symbol & Regularity of $X$ \\ \hline \noalign{\vskip 2pt}
$\overline{\rT}_{7}$ & 
\begin{tikzpicture}
\tikzstyle{every node}=[draw,solid,draw=black,fill=black,shape=circle,minimum size=.05cm,inner sep=.05cm]
\path[solid,draw=black,fill=white,thin] 
(0,0) node (root 1) {} --++(.5,0) node (root 2) {} --++(.5,.25) node (root 3) {} --++(-.5,.25) node (root 4) {} --++(-.5,0) node[double] (root 5) {}
(root 3)
\foreach \i in {6,...,8}
{
--++(.5,0) node (root \i) {}
};
\end{tikzpicture}
& $2_{32}$ & $(576, 35, 12, 6)$ \\
Extended $\overline{\rT}_7$ & 
\begin{tikzpicture}
\tikzstyle{every node}=[draw,solid,draw=black,fill=black,shape=circle,minimum size=.05cm,inner sep=.05cm]
\path[solid,draw=black,fill=white,thin] 
(0,0) node (root 1) {} --++(.5,0) node (root 2) {} --++(.5,.25) node (root 3) {} --++(-.5,.25) node (root 4) {} --++(-.5,0) node[double] (root 5) {}
(root 3)
\foreach \i in {6,...,9}
{
--++(.5,0) node (root \i) {}
};
\end{tikzpicture}
& $2_{42}$ & $(17280,56,15,6)$ \\
$\overline{\rT}_{8}$ & 
\begin{tikzpicture}
\tikzstyle{every node}=[draw,solid,draw=black,fill=black,shape=circle,minimum size=.05cm,inner sep=.05cm]
\path[solid,draw=black,fill=white,thin] 
(0,0) node[double] (root 1) {} --++(.5,0) node (root 3) {} --++(.5,0) node (root 4) {} --++(.5,0) node (root 5) {} --++(0,.5) node (root 2) {}
(root 5)
\foreach \i in {6,...,9}
{
--++(.5,0) node (root \i) {}
};
\end{tikzpicture}
& $3_{41}$ & $(2160, 64, 21, 10, 5)$ \\
$\rE_{m}$ ($m \geq 10$) & 
\begin{tikzpicture}
\tikzstyle{every node}=[draw,solid,draw=black,fill=black,shape=circle,minimum size=.05cm,inner sep=.05cm]
\path[solid,draw=black,fill=white,thin] 
(0,0) node[double] (root 1) {} --++(.5,0) node (root 3) {} --++(.5,0) node (root 4) {} --++(0,.5) node (root 2) {}
(root 4) --++(.5,0) node (root 5) {};
\path[dash pattern=on 2pt off 1pt,draw=black,fill=white,thin]
(root 5) --++(.5,0) node (root 6) {};
\path[solid,draw=black,fill=white,thin]
(root 6)
\foreach \i in {7,...,10}
{
--++(.5,0) node (root \i) {}
};
\end{tikzpicture}
& $2_{(m-4)1}$ & $(2^{m-2}, \frac{(m-1)(m-2)}{2}, 2(m-3), m-3)$ \\
\hline
\end{tabular}
\end{center}

\medskip

The most regular example in Table \ref{tab:Wythoffian-hyperbolic-honeycombs} is the honeycomb $3_{41}$ in hyperbolic $8$-space, whose $1$-skeleton is a $(2160,64,21,10,5)$-regular, $(2160,64,21,10)$-connected regular graph. Its links are successively the $2_{41}$ polytope with symmetry group the Coxeter group $E_8$, the $7$-demicube, the rectified $6$-simplex, the 5-cell prism, and the disjoint union of a vertex and a 3-simplex.

All the examples of Table \ref{tab:Wythoffian-hyperbolic-honeycombs} have an infinite facet. The first, the third, and the last one for $m = 10$, have facets of finite hyperbolic volume, hence their automorphism groups are noncocompact lattices in the isometry group of their respective hyperbolic space. They were discovered by Coxeter \cite{Coxeter1948}, forerun by Gosset's discovery \cite{Gosset1900} of the so-called \emph{uniform $k_{21}$ polytopes} and by Elte's subsequent work \cite{Elte1912}.

\begin{remark} \label{rem:limited-Wythoffian-examples}
Along the same lines as Remark \ref{rem:limited-regular-examples} and by going through the list of definite Coxeter diagrams carefully, one can see that the only polytopes of regularity level 4 or more that can be obtained via this method are those in Table \ref{tab:Wythoffian-hyperbolic-honeycombs} and the honeycomb from \S\ref{subsec:order-5-4-simplex-honeycomb}. 

We discuss next an example where arbitrarily high regularity is achieved, but at the cost of connectivity at the level of triangle links. 
\end{remark}

\subsection{An example with arbitrarily high regularity} \label{subsec:arbitrarily-high-regularity}

Even though in \S \ref{subsec:Wythoffian-polytopes} we required that all low-dimensional faces were simplices, this is not necessary in general. It is however obviously necessary that some faces are simplicial (otherwise the 1-skeleton contains no cliques). 

For any $m \geq 5$, let $\cP_m$ be the Wythoffian polytope associated with the diagram 
\begin{tikzpicture}[scale=0.7, transform shape]
\tikzstyle{every node}=[draw,solid,draw=black,fill=black,shape=circle,minimum size=.05cm,inner sep=.05cm]
\path[solid,draw=black,fill=white,thin] 
(0,0) node (root 1) {} --++(.5,0) node (root 2) {};
\path[dash pattern=on 2pt off 1pt,draw=black,fill=white,thin]
(root 2) --++(.5,0) node (root 3) {};
\path[solid,draw=black,fill=white,thin] 
(root 3) --++(.5,0) node (root 4) {} --++(.5,0) node (root 5) {};
\path[dash pattern=on 2pt off 1pt,draw=black,fill=white,thin]
(root 5) --++(.5,0) node (root 6) {};
\path[solid,draw=black,fill=white,thin]
(root 6) --++(.5,0) node (root 7) {};
\path[solid,draw=black,fill=white,thin]
(root 4) --++(0,.5) node[double] (root 0) {};
\draw [decorate,decoration={brace,amplitude=.1cm},xshift=0pt,yshift=.1cm]
(-0.1,0) -- node[draw=none,fill=none,shape=ellipse,above]{$m-1$} (1.1,0);
\draw [decorate,decoration={brace,amplitude=.1cm},xshift=0pt,yshift=.1cm]
(1.9,0) -- node[draw=none,fill=none,shape=ellipse,above]{$m-1$} (3.1,0);
\end{tikzpicture}
and let $X_m$ its 1-skeleton. 
This polytope has all $3$-faces tetrahedral, but some $4$-faces are 4-demicubes. 
Its vertex links are $(m-1)$-rectified $(2m-1)$-simplices, with $\binom{2m}{m}$ vertices and diagram
\begin{tikzpicture}[scale=0.7, transform shape]
\tikzstyle{every node}=[draw,solid,draw=black,fill=black,shape=circle,minimum size=.05cm,inner sep=.05cm]
\path[solid,draw=black,fill=white,thin] 
(0,0) node (root 1) {} --++(.5,0) node (root 2) {};
\path[dash pattern=on 2pt off 1pt,draw=black,fill=white,thin]
(root 2) --++(.5,0) node (root 3) {};
\path[solid,draw=black,fill=white,thin] 
(root 3) --++(.5,0) node[double] (root 4) {} --++(.5,0) node (root 5) {};
\path[dash pattern=on 2pt off 1pt,draw=black,fill=white,thin]
(root 5) --++(.5,0) node (root 6) {};
\path[solid,draw=black,fill=white,thin]
(root 6) --++(.5,0) node (root 7) {};
\draw [decorate,decoration={brace,amplitude=.1cm},xshift=0pt,yshift=.1cm]
(-0.1,0) -- node[draw=none,fill=none,shape=ellipse,above]{$m-1$} (1.1,0);
\draw [decorate,decoration={brace,amplitude=.1cm},xshift=0pt,yshift=.1cm]
(1.9,0) -- node[draw=none,fill=none,shape=ellipse,above]{$m-1$} (3.1,0);
\end{tikzpicture}, whose vertex links in turn are the Cartesian product of two $(m-1)$-simplices. 
By \S\ref{subsec:combining}, the 1-skeleton of the Cartesian product of two $(m-1)$-simplices is a $(2(m-1), m-2, m-3 \dots, 1)$-regular graph on $m^2$ vertices. 
Since all the $3$-faces of $\cP_m$ are simplicial, the 3-skeleta of $\cP_m$ and $X_m^\textrm{cl}$ coincide, hence $X_m$ is a $(\binom{2m}{m}, m^2, 2(m-1), m-2,m-3,\dots,1)$-regular graph. 

In this example, connectivity breaks down quickly since the link of an edge is a product of two polytopes (as the diagram becomes disconnected). The 1-skeleton of this link is then a Cartesian product of graphs, in which the sphere around any vertex is always disconnected. Thus $X_m$ has connected regularity level $2$. 

Note that the 1-skeleton of the $(m-1)$-rectified $(2m-1)$-simplex is just the Johnson graph $J(2m,m)$. Indeed, vertices of the $(m-1)$-rectified $(2m-1)$-simplex are placed in the center of the $(m-1)$-faces of a $(2m-1)$-simplex, hence correspond to subsets of $\{1, \dots, 2m\}$ of size $m$, with two vertices connected by an edge when the two corresponding subsets have $m-1$ elements in common. 

One can read from the diagram or deduce from the above combinatorial description that the sphere around an $i$-clique in the $1$-skeleton of $\cP_m$ is (as $i$ ranges from $1$ to $m+1$) successively : the Johnson graph $J(2m,m)$, the Cartesian product of two complete graphs $K_m$ on $m$ vertices, the disjoint union of two $K_{m-1}$, then $K_{m-2}$, $K_{m-3}$, and so on. In particular the links in $X^\textrm{cl}$ (!) of $2$-simplices are disconnected, while for $-1 \leq i \leq m$, $i \neq 2$, those of $i$-simplices are connected.

\section{Degree parameters of highly regular (expander) graphs} \label{sec:degree-parameters}

In this section we investigate the existence (or non-existence) of highly regular connected graphs with given degrees $(a_0,\dots,a_{n-1})$, and we present some suggestions for further research.

To facilitate the discussion we define the following sets of $n$-tuples of positive integers: 
\begin{itemize}[leftmargin=*]
\item $\HR(n)=\{\mathbf{a}=(a_0,\dots,a_{n-1}) \in \mathbb{N}^n$ s.t.\ there exists a connected $\mathbf{a}$-regular graph$\,\}$,
\item $\HR_{\infty}(n)=\{\mathbf{a}=(a_0,\dots,a_{n-1}) \in \mathbb{N}^n$ s.t.\ there exist up to isomorphism infinitely many connected $\mathbf{a}$-regular graphs$\,\}$,
\item $\HR_{\exp}(n)=\{\mathbf{a}=(a_0,\dots,a_{n-1}) \in \mathbb{N}^n$ s.t.\ there exists an infinite family of $\mathbf{a}$-regular expander graphs$\,\}$.
\end{itemize}

We also introduce the sets corresponding to the additional requirement of connectivity for the links of higher-dimensional cells.
\begin{itemize}[leftmargin=*]
\item $\HRC(n)=\{\mathbf{a}=(a_0,\dots,a_{n-1}) \in \mathbb{N}^n$ s.t.\ there exists a $\mathbf{a}$-connected regular graph$\,\}$,
\item $\HRC_{\infty}(n)=\{\mathbf{a}=(a_0,\dots,a_{n-1}) \in \mathbb{N}^n$ s.t.\ there exist up to isomorphism infinitely many $\mathbf{a}$-connected regular graphs$\,\}$,
\item $\HRC_{\exp}(n)=\{\mathbf{a}=(a_0,\dots,a_{n-1}) \in \mathbb{N}^n$ s.t.\ there exists an infinite family of $\mathbf{a}$-connected regular expander graphs$\,\}$.
\end{itemize}

\begin{remark}\label{case-n=1}
Note that $\HR(1) = \HRC(1) =\mathbb{N}$ (viewing a single vertex as a $0$-regular graph), $\HR_\infty(1) = \HRC_\infty(1)=\mathbb{N}\setminus \{0,1\}$ since every connected finite $2$-regular graph is a cycle, and $\HR_{\exp}(1) = \HRC_{\exp}(1)=\mathbb{N} \setminus \{0,1,2\}$ since for $k\geq 3$ there exists a family of $k$-regular expander graphs by a result of Pinsker \cite{Pinsker}.
\end{remark}

\subsection{Combining highly regular expanders}\label{subsec:combining}

Here we recall two constructions which allow us to combine existing parameter sets of level $n$ in order to create ``larger'' ones. 
These constructions together with Lemma \ref{graphproducts} show that a set of the form $\HR(n)$, $\HR_\infty(n)$ or $\HR_{\exp}(n)$ is infinite whenever it is non-empty. 

Let $G_i$ be an $(a_0^i,a_1^i,\cdots,a_n^i)$-regular graph with vertex set $V_i$, for $i=1,2$.

\smallskip\par\noindent
(1) {\bf Tensor product construction} (see also \cite{ChapmanLinialPeled2019}): The graph tensor product of $G_1$ and $G_2$ \cite{RusselWhitehead1912}, denoted by $G_1 \times G_2$,  has vertex set $V_1\times V_2$, and adjacency is defined by $(v_1,v_2) \sim (v'_1,v'_2)$ if and only if $v_1 \sim v'_1$ and $v_2 \sim v_2'$. It is straightforward to show that $G_1 \times G_2$ is $(a_0^1a_0^2,a_1^1a_1^2,\cdots,a_n^1a_n^2)$-regular.

\smallskip\par\noindent
(2) {\bf Cartesian product construction}: The Cartesian product of $G_1$ and $G_2$ (see e.g. \cite{Sabidussi1960}), denoted by $G_1 \square G_2$, has vertex set $V_1\times V_2$, and adjacency is defined by $(v_1,v_2) \sim (v'_1,v'_2)$ if and only if 
either $v_1 = v'_1$ and $v_2 \sim v_2'$ or $v_1\sim v_1'$ and $v_2=v'_2$. It is easy to check that if $a_k^1=a_k^2=:a_k$ for $k\in \{1,\dots,n\}$, then $G_1 \square G_2$ is 
$(a_0^1+a_0^2,a_1,\cdots,a_n)$-regular. 

\begin{remark}\label{Cheeger-Buser}
The Cheeger-Buser inequalities (\cite{Dodziuk1984,Alon-Milman-1984} show that a family of connected $k$-regular graphs $(G_n)_{n\in \NN}$ is a family of expanders if and only if their adjacency matrices $(A_n)_{n\in \NN}$ possess a uniform spectral gap $s = k-\lambda_{2,n}>0$. Here $\lambda_{2,n}$ denotes the second largest eigenvalue of $A_n$.
\end{remark}   

\begin{lemma}\label{graphproducts}
If $\mathcal{G}^1 = (G^1_i)_{i\in I}$ and $\mathcal{G}^2 = (G^2_i)_{i\in I}$ form a family of expander graphs, then so do $\mathcal{G}^1 \times \mathcal{G}^2=(G^1_{i} \times G^2_{i})_{i\in I}$ and 
$\mathcal{G}^1 \square \mathcal{G}^2 = (G^1_{i} \square G^2_{i})_{i\in I}$.
\end{lemma}
\begin{proof} We use Remark \ref{Cheeger-Buser}. Let $A_i$ be the adjacency matrix of $G_i$, for $i=1,2$. The eigenvalues of the adjacency matrix $A_1 \otimes A_2$ of the tensor product $G_1 \times G_2$ \cite{RusselWhitehead1912} are the pairwise products of the eigenvalues of $A_1$ and $A_2$. The adjacency matrix $A_{1\square 2}$ of $G_1 \square G_2$ is the Kronecker sum of $A_1$ and $A_2$, namely $A_{1\square 2} = A_1 \otimes I_{n_2} + I_{n_1} \otimes A_2$. The eigenvalues of $A_{1\square 2}$ are all of the form $\lambda_1+\lambda_2$ where $\lambda_i$ is an eigenvalue of $A_i$ for $i = 1,2$ (see \cite[Theorem 4.4.5]{HornJohnson1991}). Hence if the families $\mathcal{G}^1$ and $\mathcal{G}^2$ have uniform spectral gap, then so do $\mathcal{G}^1 \times \mathcal{G}^2$
and $\mathcal{G}^1 \square \mathcal{G}^2$, and the lemma is proved.
\end{proof}

We now discuss the behavior of $\mathbf{a}$-connected regularity under the above graph products. Note that the vertex link of a Cartesian product is never connected, so we only consider tensor products. 

\begin{lemma}\label{lemma:HRCproducts}
Let $n\geq 3$. If $G_1$ and $G_2$ are connected regular of level $n$,
then $G_1\times G_2$ is connected regular of level $n-2$. 
\end{lemma}
\begin{proof}
Note that the $j$-links of the tensor product of two graphs are the tensor products of the $j$-links of those graphs. 
By a result of Weichsel \cite{Weichsel-1962}, the tensor product of two connected graphs is connected if and only if at least one of them contains an odd cycle. 
So for given $j$, as long as (at least one of the) $j$-links contain a triangle, the $j$-link of the tensor product will be connected. 
\end{proof}

\subsection{Restrictions on the regularity degrees}\label{subsec:restrictions}
It would be interesting to describe the sets $\HR(n)$, $\HR_{\infty}(n)$ and $\HR_{\exp}(n)$, as well as their connected counterparts $\HR(n)$, $\HR_{\infty}(n)$ and $\HR_{\exp}(n)$,  precisely. 

Apart from the obvious bounds $a_i > a_{i+1}$ and the requirement that the product of any $k$ consecutive $a_i$'s must be divisible by $k!$, we do not know any necessary condition for a tuple $(a_0, \dots, a_{n-1})$ to be contained in one of these sets. 
For one thing, the constructions above show that one cannot bound one of the $a_i$ in terms of its successors. 

Zelinka \cite{Zelinka2000} showed by ad hoc methods that there exist no $(7,4)$-regular graphs. 
Note also that strongly regular graphs form a subclass of $(a,b)$-regular graphs, and it is a well-known open problem to determine the allowable parameters for strongly regular graphs. 
For example, if $G$ is a graph in which every edge is in a unique triangle, and every non-edge is a diagonal of a unique 4-cycle, then it is a relatively easy  exercise to show that $|G| \in \{3, 9, 99, 243, 6273, 494019\}$. Moreover, examples of such graphs for $|G| = 3,9$ and $243$ are given by a triangle, a toroidal 3-by-3 grid and a very interesting example related to the ternary Golay code \cite{BerlekampVanLintSeidel1973}. But is there an example with $|G| = 99\,$? Conway offered US\$1000 for an answer to this question, which is now known as `Conway's 99-graph problem' (although in fact its history goes back to Biggs in 1969).

\subsection{Graphs of arbitrarily high connected regularity level}\label{arbitrarylevel}
Here we provide two group-theoretic constructions which show that $\HRC(n)$ is infinite for all $n$. The connected regularity in both constructions follows from Lemma \ref{lem:higher-regularity}. 

\smallskip \par\noindent
{\bf Example 1:} 
The rank $n$ Coxeter group $[3,3,..,3,2]$ is finite, of order $2n!$, and isomorphic to the direct product $S_n \times C_2$. 
Now let $x_1,\dots x_n$ be the canonical involutory generators of the 
rank $n$ Coxeter group $G=[3,3,..,3,\infty]$, and let $U$ be the quotient of $G$ obtained by adjoining the single extra relation 
$(x_{n-2}x_{n-1}x_{n})^6 = 1$.  
The latter relation is equivalent to $[w,w^{x_{n-2}}] = 1$, where $w = (x_{n-1}x_{n})^2$, 
because if $(a,b,c) = (x_{n-2},x_{n-1},x_{n})$ then 
$$[(bc)^2,a(bc)^{2}a] = cbc(bca)^{4}bcbca = cbc(bca)^{6}(acb)^{2}bcbca = cbc(bca)^{6}cababac = cbc(bca)^{6}cbc.$$

Now let $\,w_2 = w^{x_{n-2}}$, $\,w_3 = w^{x_{n-2}x_{n-3}}$, and so on, up to $\,w_{n-1} = w^{x_{n-2}x_{n-3}\dots x_1}$. 
Then it is an easy exercise using the Coxeter group relations to show that the effect of conjugation of these $w_j$ by 
the generators $x_1,\dots x_n$ of $U$ is as follows:

\medskip\noindent
\hskip 36pt  $x_i$ interchanges $w_{n-1-i}$ with $w_{n-i}$, and centralises all other $w_j$, for $1 \le i \le n-2$; 

\smallskip\noindent
\hskip 36pt  $x_{n-1}$ inverts $w_1$, and interchanges  $w_{j}$ with $w_{1}^{-1}w_{j}$, for $2 \le j \le n-1$; and  

\smallskip\noindent
\hskip 36pt  $x_{n}$ inverts $w_{j}$ for all $j$. 

\medskip\noindent
Hence the elements $w_j$ generate a normal subgroup $N$ of $U$, with quotient $U/N \cong S_n \times C_2$. 
Also $[w_1,w_2] = [w,w^{x_{n-2}}] = 1$, and repeated conjugation of this by the generators $x_i$ show that 
every two of the $w_i$ commute, and therefore $N$ is abelian. 

Moreover, the effect of conjugation of the $w_j$ by the subgroup of $U$ generated by $x_1,\dots x_{n-1}$  
is equivalent to the action of $S_n$ on its $(n-1)$-dimensional augmentation module (consisting of all $n$-vectors 
$(v_1,v_2,\dots,v_n)$ with vector-sum $0$), and as this module is irreducible over $\R$ (and hence over $\mathbb{Q}$), 
it follows that $N$ is free abelian of rank $n-1$. 

Next, for any positive integer $k$, factor out the characteristic subgroup $N_k$ 
of $N$ generated by the $k$-th powers of $w = (x_{n-1}x_{n})^2$ and their conjugates. 
Then the resulting finite quotient $U/N_k$ of $G$ has order $2(n!)k^{n-1},$ and is the automorphism group of 
a regular polytope of rank $n$ with type $[3,3,..,3,2k]$, 
and gives rise to a highly regular graph of level $n-1$ with parameters $(nk,(n-1)k,(n-2)k,\dots,3k,2k)$.

\smallskip\par\noindent
{\bf Example 2:} Conder, Hubard and O'Reilly-Regueiro \cite{CHO2020} recently devised a construction in order to produce the first concrete examples of chiral (but otherwise maximally symmetric) polytopes of arbitrarily large rank, showing also that for every integer $n \ge 5$, all but finitely many of the alternating groups $A_k$ and symmetric groups $S_k$ are the automorphism groups of regular polytopes of rank $n$ and type $[3,3,..,3,m]$ for some $m$ (dependent on $k$ and $n$), with simplicial facets. 
(This can be achieved by constructing suitable homomorphisms from the rank $n$ Coxeter group $[3,3,\dots,3,\infty]$ onto $A_n$ and $S_n$ for all sufficiently large $n$.) 
The parameters $(a_0,\dots,a_{n-2})$ for the resulting highly regular graphs, however, involve very large integers and reveal no obvious recurring patterns. 

\subsection{Related work by Friedgut and Iluz} \label{subsec:related}
During the write-up of this paper it was brought to the authors' attention that Friedgut and Iluz, in work in preparation, have obtained related results. They observed that the Coxeter group $\rH_5$ leads to the construction of $(120,12,5,2)$-regular graphs, and Friedgut had presented this work at Oberwolfach in April 2019, but with no mention of the expansion of those graphs. They also informed us they have a method to show that $\HRC_{\infty}(n)$ and even $\HRC_{\exp}(n)$ are infinite (compare with Theorem \ref{thm:levels}(c)).

\subsection{Open problems}\label{subsec:problems}
We conclude with the following natural problems:

\medskip\par\noindent
{\bf Problem A:} Consider the following diagram of inclusions: \medskip
\begin{center}
\begin{tikzcd}
   \HRC_{\exp}(n) \arrow[r,hook] \arrow[d,hook]
    & \HRC_{\infty}(n) \arrow[r,hook] \arrow[d,hook] & \HRC(n) \arrow[d,hook] \\
  \HR_{\exp}(n) \arrow[r,hook] & \HR_{\infty}(n) \arrow[r,hook] & \HR(n) 
\end{tikzcd}
\medskip
\end{center}
Are any of these inclusions strict for $n>1$? (See Remark \ref{case-n=1} for the case $n=1$.)

\medskip\par\noindent
{\bf Problem B:} For $n>1$ describe the above six sets as subsets of $\mathbb{N}^n$. 

\section{Dedication to John Conway and Ernest Vinberg}\label{sec:CV}

This paper is dedicated to John Conway and Ernest Vinberg, for their phenomenal insights and outstanding contributions 
in the fields of algebra, combinatorics and geometry.   Both of them died in 2020, casualties of the Covid-19 virus. 
Their work has been inspirational to us and to hundreds of other mathematicians worldwide. 

John Conway is perhaps best known for his contributions to combinatorial game theory, especially the `Game of Life', and for the discovery of three of the sporadic finite simple groups. But he also made fundamental discoveries across a very wide range of other topics, including knots, lattices, numbers, polyhedra and tilings. 
Ernest Vinberg is best known for his work on discrete subgroups of Lie groups and representation theory. 
He introduced Vinberg's algorithm for finding a fundamental domain of a hyperbolic reflection group, and he developed some beautiful theory of the arithmetic nature of co-finite hyperbolic Coxeter groups and the combinatorial-metric structure of their Coxeter polyhedra in terms of the Gram matrix. 
(Also incidentally, Conway was a great admirer of Coxeter, whose groups play a key role in this paper, 
and he used Vinberg's algorithm to describe the automorphism group of the 26-dimensional even unimodular Lorentzian 
lattice II$_{25,1}$ in terms of the Leech lattice.)  

Conway had a life-long interest in highly symmetric objects, and Vinberg made great contributions to the theory and applications of Coxeter groups. This paper which combines these two threads of their research serves as a tribute to them both. 

\bibliographystyle{alpha}
\bibliography{bibliography}

\begin{thebibliography}{CHOR20}

\bibitem[AM85]{Alon-Milman-1984}
N.~Alon and V.~D. Milman.
\newblock $\lambda_1$, isoperimetric inequalities for graphs, and
  superconcentrators.
\newblock {\em J. Combin. Theory Ser. B}, 38(1):73--88, 1985.

\bibitem[Bd04]{BenoistdelaHarpe2004}
Y.~{Benoist} and P.~{de la Harpe}.
\newblock {Adh\'erence de Zariski des groupes de Coxeter.}
\newblock {\em {Compos. Math.}}, 140(5):1357--1366, 2004.

\bibitem[BLS73]{BerlekampVanLintSeidel1973}
E.R. Berlekamp, J.H.~Van Lint, and J.J. Seidel.
\newblock A strongly regular graph derived from the perfect ter- nary golay
  code.
\newblock In J.N.~Shrivastava (Ed.), editor, {\em A Survey in Combinatorial
  Theory}, pages 25--30. North-Holland, Amsterdam, 1973.

\bibitem[Bor60]{Borel1960}
A.~Borel.
\newblock Density properties of certain subgroups of semisimple groups without
  compact components.
\newblock {\em Ann. of Math.}, 72 (2):179--188, 1960.

\bibitem[{Bou}07]{Bourbaki2007}
N.~{Bourbaki}.
\newblock {\em {\'El\'ements de math\'ematique. Groupes et alg\`ebres de Lie.
  Chapitres 4 \`a 6.}}
\newblock Berlin: Springer, reprint of the 1968 original edition, 2007.

\bibitem[CG65]{Conway-Guy-1965}
J.H. Conway and M.J.T. Guy.
\newblock Four-dimensional archimedean polytopes.
\newblock {\em Proceedings of the Colloquium on Convexity at Copenhagen}, pages
  38--39, 1965.

\bibitem[CHOR20]{CHO2020}
M.~Conder, I.~Hubard, and E.~O’Reilly-Regueiro.
\newblock Construction of chiral polytopes of large rank with alternating or
  symmetric automorphism group, 2020.
\newblock To appear.

\bibitem[CLP20]{ChapmanLinialPeled2019}
M.~{Chapman}, N.~{Linial}, and Y.~{Peled}.
\newblock Expander graphs -- both local and global.
\newblock {\em Combinatorica, to appear}, 2020.

\bibitem[CM05]{ConderMaclachlan2005}
M.~{Conder} and C.~{Maclachlan}.
\newblock {Compact hyperbolic 4-manifolds of small volume.}
\newblock {\em {Proc. Am. Math. Soc.}}, 133(8):2469--2476, 2005.

\bibitem[{Cox}34]{Coxeter1934}
H.~S.~M. {Coxeter}.
\newblock {Wythoff's construction for uniform polytopes.}
\newblock {\em {Proc. Lond. Math. Soc. (2)}}, 38:327--339, 1934.

\bibitem[{Cox}40]{Coxeter1940}
H.~S.~M. {Coxeter}.
\newblock {Regular and semi-regular polytopes. I.}
\newblock {\em {Math. Z.}}, 46:380--407, 1940.

\bibitem[{Cox}48]{Coxeter1948}
H.~S.~M. {Coxeter}.
\newblock {\em {Regular polytopes}}.
\newblock {London: Methuen \& Co., Ltd.}, 1948.

\bibitem[{Cox}56]{Coxeter1956}
H.~S.~M. {Coxeter}.
\newblock {Regular honeycombs in hyperbolic space}.
\newblock In {\em {Proceedings of the International Congress of Mathematicians
  1954.}}, volume III. Groningen: Erven P. Noordhoff N. V.; Amsterdam:
  North-Holland Publishing Co., 1956.

\bibitem[{Cox}85]{Coxeter1985}
H.~S.~M. {Coxeter}.
\newblock {Regular and semi-regular polytopes. II.}
\newblock {\em {Math. Z.}}, 188:559--591, 1985.

\bibitem[DK17]{Dinur-Kaufman-2017}
Irit Dinur and Tali Kaufman.
\newblock High dimensional expanders imply agreement expanders.
\newblock {\em Proc. 58th IEEE Symp. on Foundations of Comp. Science (FOCS)},
  pages 974--985, 2017.

\bibitem[Dod84]{Dodziuk1984}
J.~Dodziuk.
\newblock Difference equations, isoperimetric inequality and transience of
  certain random walks.
\newblock {\em Trans. Amer. Math. Soc.}, 284(2):787--794, 1984.

\bibitem[{Elt}12]{Elte1912}
E.L. {Elte}.
\newblock {\em {The semiregular polytopes of the hyperspaces.}}
\newblock Dissert. Univ. Groningen. {Groningen: Hoitsema}, 1912.

\bibitem[FI20]{FriedgutIluz2020}
E.~Friedgut and Y.~Iluz.
\newblock Hyper-regular graphs and high dimensional expanders, 2020.
\newblock In preparation.

\bibitem[Gar73]{Garland-1973}
H.~Garland.
\newblock p-adic curvature and the cohomology of discrete subgroups of p-adic
  groups.
\newblock {\em The Annals of Mathematics}, pages 97(3):375--423, 1973.

\bibitem[{Gos}00]{Gosset1900}
T.~{Gosset}.
\newblock {On the regular and semi-regular figures in space of $n$ dimensions.}
\newblock {\em {Messenger Math.}}, 29:43--48, 1900.

\bibitem[HJ91]{HornJohnson1991}
R.A. Horn and C.R. Johnson.
\newblock {\em Topics in Matrix Analysis.}
\newblock {Cambridge University Press, Cambridge}, 1991.

\bibitem[HLW06]{HooryLinialWigderson2006}
S.~{Hoory}, N.~{Linial}, and A.~{Wigderson}.
\newblock Expander graphs and their applications.
\newblock {\em Bull. Amer. Math. Soc.}, 43(04):439--562, 2006.

\bibitem[KM17]{Kaufman-Mass-2017}
T.~Kaufman and D.~Mass.
\newblock High dimensional random walks and colorful expansion.
\newblock {\em ITCS}, page 4:1–4:27, 2017.

\bibitem[LM91]{LubotzkyMann1991}
A.~{Lubotzky} and A.~{Mann}.
\newblock {On groups of polynomial subgroup growth.}
\newblock {\em {Invent. Math.}}, 104(3):521--533, 1991.

\bibitem[{Lub}96]{Lubotzky1996}
A.~{Lubotzky}.
\newblock {Free quotients and the first Betti number of some hyperbolic
  manifolds.}
\newblock {\em {Transform. Groups}}, 1(1-2):71--82, 1996.

\bibitem[Lub18]{Lubotzky-2018}
A.~Lubotzky.
\newblock High dimensional expanders.
\newblock {\em Proceedings of the International Congress of
  Mathematicians—Rio de Janeiro 2018}, pages Vol. I. Plenary lectures,
  705–730, 2018.

\bibitem[Mal40]{Malcev1940}
A.I. Malcev.
\newblock On isomorphic representations of infinite groups by matrices.
\newblock {\em Mat. Sb.}, 8:405--422, 1940.

\bibitem[MS02]{McMullenSchulte2002}
P.~{McMullen} and E.~{Schulte}.
\newblock {\em {Abstract regular polytopes.}}, volume~92.
\newblock Cambridge: Cambridge University Press, 2002.

\bibitem[NV02]{NoskovVinberg2002}
G.~Noskov and E.B. Vinberg.
\newblock Strong tits alternative for subgroups of coxeter groups.
\newblock {\em Lie Theory}, 12:259--264, 2002.

\bibitem[Pin73]{Pinsker}
M.~S. Pinsker.
\newblock On the complexity of a concentrator.
\newblock {\em 7th International Telegraffic Conference}, page 318/1–318/4,
  1973.

\bibitem[RW12]{RusselWhitehead1912}
B.~Russell and A.N. Whitehead.
\newblock {\em Principia Mathematica}, volume~2.
\newblock {Cambridge University Press, Cambridge}, 1912.

\bibitem[Sab60]{Sabidussi1960}
G.~Sabidussi.
\newblock Graph multiplication.
\newblock {\em Math. Z.}, 72:446--457, 1960.

\bibitem[{Sal}17]{SalehiGolsefidy2017}
A.~{Salehi Golsefidy}.
\newblock {Super-approximation. I: \(\mathfrak{p}\)-adic semisimple case.}
\newblock {\em {Int. Math. Res. Not.}}, 2017(23):7190--7263, 2017.

\bibitem[{Sal}19]{SalehiGolsefidy2019}
A.~{Salehi Golsefidy}.
\newblock {Super-approximation. II: The \(p\)-adic case and the case of bounded
  powers of square-free integers.}
\newblock {\em {J. Eur. Math. Soc.}}, 21(7):2163--2232, 2019.

\bibitem[{Sch}83]{Schlegel1883}
V.~{Schlegel}.
\newblock {Theorie der homogen zusammengesetzten Raumgebilde.}
\newblock {Leipzig: Engelmann. Nova acta acad. Leop.-Carol. XLIV.}, 1883.

\bibitem[{Tak}77]{Takeuchi1977}
K.~{Takeuchi}.
\newblock Arithmetic triangle groups.
\newblock {\em J. Math. Soc. Japan}, 29(1):91--106, 01 1977.

\bibitem[{Tit}61]{Tits1960}
J.~{Tits}.
\newblock {Groupes et g\'eom\'etries de Coxeter}.
\newblock {Bures-sur-Yvette: I.H.E.S.}, Notes polycopiées, 1961.

\bibitem[Wei62]{Weichsel-1962}
P.~M. Weichsel.
\newblock The kronecker product of graphs.
\newblock {\em American Mathematical Society}, pages 13:37--52, 1962.

\bibitem[Zel00]{Zelinka2000}
B.~Zelinka.
\newblock Locally regular graphs.
\newblock {\em Math. Bohem}, 125:481--484, 2000.

\end{thebibliography}

\end{document}